\documentclass[table]{amsart}
\usepackage[T1]{fontenc}
\usepackage[utf8]{inputenc}

\usepackage{amssymb}
\usepackage{hyperref}
\usepackage{tikz-cd}
\usepackage{tkz-graph}
\usepackage[all]{xy}
\usepackage{placeins}
\usepackage{color}
\usepackage{booktabs}
\usepackage{subcaption}
\definecolor{lightgray}{gray}{0.9}

\theoremstyle{definition}
\newtheorem{definition}{Definition}[section]
\newtheorem{example}[definition]{Example}
\newtheorem{remark}[definition]{Remark}

\theoremstyle{plain}
\newtheorem{theorem}[definition]{Theorem}
\newtheorem{lemma}[definition]{Lemma}
\newtheorem{proposition}[definition]{Proposition}
\newtheorem{corollary}[definition]{Corollary}

\DeclareMathOperator{\rk}{rk}

\DeclareMathOperator{\id}{id}

\DeclareMathOperator{\Hom}{Hom}
\DeclareMathOperator{\NS}{NS}
\DeclareMathOperator{\Aut}{Aut}

\DeclareMathOperator{\End}{End}

\DeclareMathOperator{\res}{res}

\DeclareMathOperator{\sign}{sign}
\DeclareMathOperator{\Trans}{T}

\DeclareMathOperator{\rad}{rad}

\newcommand{\QQ}{\mathbb{Q}}
\newcommand{\FF}{\mathbb{F}}
\newcommand{\NN}{\mathbb{N}}
\newcommand{\RR}{\mathbb{R}}
\newcommand{\CC}{\mathbb{C}}
\newcommand{\ZZ}{\mathbb{Z}}
\newcommand{\PP}{\mathbb{P}}
\renewcommand{\O}{\mathcal{O}}

\makeatletter
\newcommand*{\defeq}{\mathrel{\rlap{%
                     \raisebox{0.3ex}{$\m@th\cdot$}}%
                     \raisebox{-0.3ex}{$\m@th\cdot$}}%
                     =}
\makeatother
\newcounter{counter}       
\newcommand{\upperRomannumeral}[1]{\setcounter{counter}{#1}\mbox{\Roman{counter}}}
\newlength{\myleftlen}

\title[Non-symplectic automorphisms of high order on K3 surfaces]{The classification of purely non-symplectic automorphisms of high order on K3 surfaces}
\author{Simon Brandhorst}
\address{Insitut für Algebraische Geometrie, Leibniz Universität Hannover,
	Welfengarten 1, 30167 Hannover, Germany}
\email{sbrandhorst@web.de}
\date{April, 2015}
\begin{document}
\begin{abstract}
An automorphism of order $n$ of a K3 surface is called purely non-symplectic if it multiplies the holomorphic symplectic form by a primitive $n$-th root of unity. 
We give the classification of purely non-symplectic automorphisms with $\varphi(n)\geq 12$ where $\varphi$ denotes the Euler totient function.
\end{abstract}
\maketitle
\tableofcontents
\section{Introduction}
A complex K3 surface is a smooth, compact, complex surface with vanishing irregularity and trivial canonical bundle. It is not necessarily algebraic.\\

An automorphism $f$ of a K3 surface $X$ is called symplectic if it acts trivially on the global holomorphic 2-forms, $f^*|H^0(X,\Omega_X^2)=\id$, and non-symplectic otherwise. Furthermore we call $f$ purely non-symplectic if all non-trivial powers are non-symplectic.
Note that K3 surfaces admitting a non-symplectic automorphism of finite order are always algebraic \cite[3.1]{nikulin:auto}.
Finite symplectic group actions on K3 surfaces are fully classified in \cite{hashimoto:symplectic} (which builds on \cite{nikulin:auto,mukai:symplectic_automorphism_groups,kondo:symplectic_automorphism_groups,xiao:symplectic}). However, for non-symplectic automorphisms beyond the case of prime order \cite{artebani_sarti_taki:non-symplectic} most results are partial and there is only a classification of automorphisms of fixed order with extra conditions on their fixed locus or their action on divisors.

Let $X_i$ be K3 surfaces and $G_i\subseteq \Aut(X_i)$ groups of automorphisms for $i \in \{1,2\}$. We say that the pair $(X_1,G_1)$ is equivalent to $(X_2,G_2)$ if there is an isomorphism
$F: X_1 \rightarrow X_2$ such that $F G_1 F^{-1} \defeq \{F \circ g \circ F^{-1} | \} = G_2$.
The main result of this note is the following classification.
\begin{theorem}\label{thm:pair_classification}
	Let $X$ be a K3 surface and $\ZZ/n\ZZ\cong G \subseteq \Aut(X)$ a purely non-symplectic subgroup with $\varphi(n)\geq 12$, where $\varphi$ is the Euler totient function. 
	All pairs $(X,G)$ up to isomorphism are listed in Table \ref{tbl:real_dets2} where $G$ is generated by $f$.
\end{theorem}

Given two explicit K3 surfaces it can be hard to decide whether they are isomorphic or not, as it may be difficult to write down an isomorphism.
A notable exception consists in singular K3 surfaces, i.e. with Picard number $20$, which correspond up isomorphism to their oriented transcendental lattices as can be shown with the strong Torelli theorem for K3 surfaces.

The second main result can be seen as an analogue of this for K3 surfaces with a non-symplectic automorphism. 
It allows to decide effectively if two K3 surfaces (satisfying the hypothesis of the theorem) are abstractly isomorphic without having to give the isomorphism.
To formlate it we need a little more terminology.
Recall that the intersection pairing turns the Neron-Severi group $\NS(X)$ into a lattice. Its dual lattice is denoted by $\NS(X)^\vee$ and the finite quotient group $\NS(X)^\vee/\NS(X)$ is the discrimiant group. 
The transcendental lattice of $X$ is denoted by $\Trans(X)$.
\begin{theorem}\label{thm:main}
 Let $X_i, i=1,2$ be complex K3 surfaces and $f_i \in\Aut(X_i)$ automorphisms with $f_i^*(\omega_i)=\zeta_n\omega_i$ where $\zeta_n=\exp(2\pi/n)$, $n\in \NN$ and $\CC \omega_i = H^0(X_i,\Omega_{X_i}^2)$ such that $\rk \Trans(X_i)=\varphi(n)$.
 Let $I_i$ be the kernel of the natural map given by
 \[\mathbb{Z}[\zeta_n] \rightarrow O(\NS(X_i)^\vee \!/\NS(X_i)),\qquad \zeta_n\mapsto f_i^*.\]
 If $f_1$ is of finite order, then $X_1 \cong X_2$ if and only if $I_1=I_2$.
\end{theorem}

As an application we obtain that the three K3 surfaces with a purely non-symplectic automorphism of order $21$ (see Table \ref{tbl:real_dets2}) are in fact abstractly isomorphic even though the three automorphisms (and the equations describing the K3 surfaces) are quite different.

\section{Complex K3 surfaces}
In this section we recall some basic facts about complex K3 surfaces and non-symplectic automorphisms. The main reference is \cite[VIII §11]{BHPV:compact_complex_surfaces}.\\

Let $X$ be a complex K3 surface. Its second singular cohomology $H^2(X,\mathbb{Z})$ equipped with the cup product is an even unimodular lattice of signature $(3,19)$. Such a lattice is unique up to isometry. We call it the $K3$ lattice and denote it by $L_{K3}$. 
Definitions and properties of lattices are given in the next section.
By the Hodge decomposition
\[H^2(X,\mathbb{Z})\otimes \mathbb{C} \cong H^2(X,\mathbb{C})=H^{2,0}(X) \oplus H^{1,1}(X) \oplus H^{0,2}(X)\]
where $H^{i,j}(X)\cong H^j(X,\Omega_X^i)$, $H^{i,j}(X)=\overline{H^{j,i}(X)}$ and $H^{1,1}(X)=(H^{2,0}(X) \oplus H^{0,2}(X))^\perp$ has signature $(1,19)$.\\

Conversely, the Hodge structure of a K3 surface determines it up to isomorphism as is reflected by the Torelli theorems. 
\begin{theorem}\cite[VIII 11.1]{BHPV:compact_complex_surfaces}
 Let $X,Y$ be complex K3 surfaces and 
 \[f: H^2(X,\mathbb{Z}) \rightarrow H^2(Y,\mathbb{Z})\]
 an isometry of lattices whose $\mathbb{C}$-linear extension maps $H^{2,0}(X)$ to $H^{2,0}(Y)$. Then $X\cong Y$.
 If moreover $f$ maps ample classes on $X$ to ample classes on $Y$, then $f=F^*$ for a unique isomorphism
 $F: Y \rightarrow X$. 
\end{theorem}

By Lefschetz' Theorem on $(1,1)$-classes we can recover the Néron-Severi group from the Hodge structure as 
\[\NS(X)= H^{2,0}(X)^\perp\cap H^2(X,\mathbb{Z}).\]
Its rank $\rho$ is called the Picard number of $X$. 
The \emph{transcendental lattice} is defined as the smallest primitive sublattice $\Trans(X)\subseteq H^2(X,\mathbb{Z})$ whose complexification contains $H^{2,0}(X) \subseteq \Trans(X)\otimes \mathbb{C}$. 
The surface $X$ is projective if and only if $\NS(X)$ has signature $(1,\rho-1)$.
In this case $\Trans(X)=\NS(X)^\perp$.
If we consider just a single K3 surface $X$, we will usually omit the $(X)$ from notation and just write $H^{i,j}$, $\NS$, $T$. From now on all K3 surfaces are assumed to be projective. 

If $\delta \in \NS$ has self intersection $-2$, then $\delta$ or $-\delta$ is effective. 
The set of \emph{roots} is
\[\Delta_{\NS} = \{\delta \in \NS | \delta^2 = -2\}.\]
The positive cone $P$ is defined as connected component of the light cone $\{x \in L \otimes \RR | x^2 > 0\}$ that contains the ample cone. The connected components of the set \[P \setminus \bigcup_{\delta \in \Delta}\delta^\perp\]
are called the \emph{chambers} of $L$. One of the chambers is the ample cone.
Each root $\delta$ defines a reflection $r_\delta \in O(L)$ by 
$r_\delta(x)=x + (x.\delta) \delta$ along the hyperplane $\delta^\perp$.
The \emph{Weyl Group} $W(\NS)$ is the subgroup of $O(\NS)$ generated by reflections. It acts transitively on the set of chambers. Denote by $O^+(\NS)$ the subgroup of $O(NS)$ preserving the positive cone and by $\Gamma(\NS)$ the subgroup preserving the ample cone. Then
$\Gamma(\NS)$ is isomorphic to $O^+(L)/W(\NS)$ and the strong Torelli theorem provides that this group is up to finite index the automorphism group of $X$. We define the same notions in an analogous way for any
hyperbolic lattice $L$ in place of $\NS$ - after the choice of a chamber in place of the ample cone.

\section{Lattices}
 In this section we recall basic facts about lattices and introduce the necessary notation.
 Our main references are \cite{nikulin:quadratic_forms, conway_sloane:sphere_packings}.\\

 A \emph{lattice} is a finitely generated free abelian group $L\cong \mathbb{Z}^n$ together with a non-degenerate bilinear pairing
 \[L \times L \mapsto \mathbb{Z}, \qquad (x,y)\mapsto x.y\]
 It is called \emph{even} if $x^2 \defeq x.x \in 2\mathbb{Z}$ for all $x\in L$. 
 The pairing induces an isomorphism 
 \[\Hom_\mathbb{Z}(L,\mathbb{Z})\cong L^\vee\defeq\{x \in L \otimes \mathbb{Q} \mid x.L \subseteq \mathbb{Z}\}\]
 with the \emph{dual lattice} $L^\vee$. The quotient $L^\vee \!/L=:D_L$ is finite, abelian and called \emph{discriminant group}. Its order is equal to $|\det L|$.
 If $D$ is a finite abelian group, the minimum number of generators of $D$ is called the \emph{length} $l(D)$ of $D$. Note that
 $l(D_L)\leq \rk L$. 
 If $D_L=0$, or equivalently $|\det L|=1$, we call the lattice $L$ \emph{unimodular}, and if $nD_L=0$ for $n\in \NN$, we call $L$ \emph{$n$-elementary}.
 The discriminant group is equipped with a fractional form 
 \[b: D_L \times D_L \rightarrow \mathbb{Q}/\mathbb{Z}, \quad (\overline{x},\overline{y}) \mapsto x.y + \mathbb{Z}. \]
 On an even lattice there is the \emph{discriminant form} $q$ given by
 \[q: D_L \rightarrow \mathbb{Q}/2\mathbb{Z}, \qquad \overline{x}\mapsto x^2 + 2\mathbb{Z}.\]
 The discriminant group decomposes as an orthogonal direct sum of $p$-groups
 \[D_L\cong \bigoplus_p (D_L)_p.\]
 Note that, by polarization, $b|(D_L)_p^2$ and $q|(D_L)_p$ carry the same information for odd primes $p\neq2$.
 If $(D_L)_p$ is an $\mathbb{F}_p$-vector space, then $q|(D_L)_p$ takes values in $\tfrac{2}{p}\mathbb{Z}/2\mathbb{Z}$.
 A subgroup $S\subseteq D_L$ is called \emph{totally isotropic} if $q|_{D_L}=0$. Totally isotropic subspaces correspond bijectively to even overlattices.\\
 
 We say two lattices $M,N$ are in the same \emph{genus} if $N\otimes_\mathbb{Z} \mathbb{Z}_p\cong M\otimes_\mathbb{Z} \mathbb{Z}_p$ are isometric over the $p$-adic integers for all primes $p$ and $N\otimes_\mathbb{Z} \mathbb{R} \cong M\otimes_\mathbb{Z} \mathbb{R}$ over the real numbers. We use the Conway Sloane \cite[Chapter 15]{conway_sloane:sphere_packings} genus symbols to describe a genus.
 
 \begin{theorem}\cite[1.9.4]{nikulin:quadratic_forms}
  The signature $(n_+,n_-)$ and discriminant form $q$ determine the genus of an even lattice and vice versa.  
 \end{theorem}
 
 \textbf{Notation.} 
 We denote by $A_n, D_n, E_i$, $i=6,7,8$ the negative definite root lattices with the respective Dynkin diagram.
 The unique unimodular lattice of signature $(1,1)$ is called a hyperbolic plane and denoted by $U$.
 The orthogonal direct sum of two lattices $L_1,L_2$ is denoted by $L_1 \oplus L_2$ and the lattice $L$ with bilinear form rescaled by $a \in \ZZ$ by $L(a)$. The group of isometries of a lattice $L$ is denoted by $O(L)$ and the group of isometries of the discriminant form $q_L$ of $L$ by $O(q_L)$. If $N \subseteq L$ is a subset,
 then $N^\perp$ is its orthogonal complement in $L$.
 If $f$ is an endomorphism of a free abelian group or vector space $L$ we denote by $\chi_f$ its characteristic polyomial, by $\mu_f$ its minimal polynomial and for a subset $N \subseteq L$ by $fN \defeq f(N)$ its image under $f$.

\subsection{Primitive embeddings}
In this section we review the theory of primitive embeddings of lattices in the language of discrimiant forms
as developed in \cite{nikulin:quadratic_forms}.

\begin{definition}
An embedding of lattices $i:M\hookrightarrow L$ is called \emph{primitive} if the cokernel is free. 
\end{definition}

\begin{example}
 By Lefschetz theorem on $(1,1)$-classes $\NS(X) \hookrightarrow H^2(X,\ZZ)$ is a primitive embedding.
\end{example}

We call two primitive embeddings $i,j: M\hookrightarrow L$ isomorphic if there is a commutative diagram with
$f \in O(L)$.
\[
\begin{tikzcd}
M \arrow[hook]{r}{i}\arrow[hook]{rd}{j} &L \arrow{d}{f}\\
&L
\end{tikzcd}
\]

We shall say that $S$ \emph{embeds uniquely} into $L$ if all primitive embeddings are isomorphic.
A (weakened) criterion for this to happen is given in the next theorem.  
\begin{theorem}\label{thm:unique embedding}\cite[Theorem 1.14.4]{nikulin:quadratic_forms}\cite[2.8]{morrison:large_picard_number}
 Let $M$ be an even lattice of signature $(m_+,m_-)$ and $L$ an even unimodular lattice of signature $(l_+,l_-)$.
 If $l(D_M)+2\leq \rk L - \rk M$ and $l_+>m_+,l_->m_-$, then
 there is a unique primitive embedding of $M$ into $L$.
\end{theorem}

We mention the related
\begin{theorem}\label{thm:orthogonal_group_surjective}\cite[Theorem 1.14.2]{nikulin:quadratic_forms}
	Let $M$ be an even, indefinite lattice such that $\rk M\geq 2+ l(D_M)$, then the genus of $M$ contains only one class, and the homomorphism $O(M)\rightarrow O(q_M)$ is surjective.
\end{theorem}

A different perspective on primitive embeddings is given by the concept of primitive extensions.
Let $i:M\hookrightarrow L$ be a primitive embedding. Then $N\defeq M^\perp$ is also a primitive sublattice and we call
$M\oplus N \hookrightarrow L$
a \emph{primitive extension}. A \emph{glue map} is a map $\phi$ defined on certain subgroups 
\[D_{M} \supseteq G_M \xrightarrow[\phi]{\sim} G_N \subseteq D_{N},\]
with the extra condition that $q_M=-q_N \circ \phi$. 
\begin{theorem}\cite[Prop 1.15.1]{nikulin:quadratic_forms}
	There is a one to one correspondence 
	\[
	\left\{ 
	\begin{matrix}
	\mbox{Primitive extensions } \\
	M \oplus N \hookrightarrow L
	\end{matrix}
	\right\}
	\stackrel{1 : 1}{\longleftrightarrow}
	\left\{ 
	\begin{matrix}
	\mbox{Glue maps} \\
	D_{M} \supseteq G_M \xrightarrow[\phi]{\sim} G_N \subseteq D_{N}
	\end{matrix}
	\right\}.
	\]
\end{theorem}
This correspondence arises as follows.
Given a glue map $\phi$, we define \emph{the glue}
\[G_\phi:=\{x+\phi(x) |x\in G_M\} \subseteq D_{M} \oplus D_{N}\] 
as the graph of $\phi$.
By construction, $G_\phi$ is a totally isotropic subspace of $D_M \oplus D_N$. Hence, we can define a lattice $L = M \oplus_{\phi} N$ via 
\[L/(M\oplus N) = G_\phi.\]
The reader may check that $M \oplus N \hookrightarrow L$ is indeed a primitive extension.\\

Conversely, given a primitive extension as above, we interpret the totally isotropic subspace $L/(M\oplus N)$ as the defining graph $G_\phi$ of a glue map $\phi$. Explicitly this is given by the isomorphisms 
 \begin{equation}\label{eq:glue H isos}
 G_M \cong L/(M\oplus N) = G_\phi \cong G_N.
 \end{equation}
The glue map is defined on the spaces $G_M:= p_M(L)/M$ and $G_N:= p_N(L)/N$ where 
\[p_M: M^\vee \oplus N^\vee \rightarrow M^\vee \mbox{ and } p_N: M^\vee \oplus N^\vee \rightarrow N^\vee\] 
are the orthogonal projections.

With the glue map $\phi$ understood, we will often drop it from notation and simply denote by $G=L/(M \oplus N)$ the glue of a primitive extension. In this situation we say that $M$ and $N$ are glued along $|G|$.\\

There is the following constraint on the size of the glue:

\begin{lemma}\label{lem:glue estimate}
 \[|D_{N}/G_N| \cdot |D_{M}/G_M| = \det L \]
\end{lemma}

\begin{proof}
We divide the standard formula 
 \[\det M \det N = [L:M \oplus N]^2 \det L \]
by $[L:M\oplus N]^2$.  The isomorphisms (\ref{eq:glue H isos}) provide
$[L:M \oplus N] = |G_N|$. The lemma follows with $|\det M|=D_M$ and the same reasoning for $N$.  
\end{proof}

\begin{lemma} \label{lem:structure_glue_quotient}
Let $N \hookrightarrow L$ be a primitive embedding.
Then there is a surjection $D_L \twoheadrightarrow D_N/G_N$.
\end{lemma}
\begin{proof}
We have the following induced diagram with exact rows
\begin{displaymath}
\xymatrix{
0 \ar[r] & L \ar[r] \ar@{->>}[d] & L^{\vee} \ar[r] \ar@{->>}[d] & L^{\vee}/L = D_L \ar[r] \ar@{-->>}[d] & 0 \\
0 \ar[r] & p_N(L) \ar[r] & N^{\vee} \ar[r] & N^{\vee}/p_N(L) = D_N/G_N \ar[r] & 0
}
\end{displaymath}
where the primitivity of $N \hookrightarrow L$ gives the surjectivity of the central vertical arrow. The commutativity of the diagram then implies the desired surjection.
\end{proof}

\subsection{Extending isometries}
In this section we investigate the interplay between primitive extensions and isometries.

\begin{example}
 If $X$ is a complex K3 surface, then 
 \[\NS \oplus T \hookrightarrow H^2(X,\mathbb{Z})\]
 is a primitive extension. Let $f$ be an automorphism of $X$. It acts (by pullback) on $\NS$, $T$, $H^2(X,\mathbb{Z})$, $D_{\NS}=G_{\NS}\cong G \cong G_T=D_T$ in a compatible way. If confusion is unlikely, we will denote all these actions by $f$ or for example by $f|D_T$.
\end{example}

Clearly, an isometry $f=f_M \oplus f_N$ defined on $M \oplus N$ extends to a primitive extension $L$ if and only if $f(L/(M \oplus N))=L/(M\oplus N)$, i.e. 
$f(G_\phi)=G_\phi$.
In other words $f_M|D_M$ preserves $G_M$, $f_N|D_N$ preserves $G_N$ and $\phi \circ f_M=f_N \circ \phi$.
	\[
	\left\{ 
	\begin{matrix}
	\mbox{Primitive extensions } M\oplus N \hookrightarrow L \\
	 \mbox{ such that } f_M \oplus f_N \mbox{ extends to } L
	\end{matrix}
	\right\}
	\stackrel{1 : 1}{\longleftrightarrow}
	\left\{ 
	\begin{matrix}
	\mbox{Glue maps } \phi: G_M \xrightarrow{\sim} G_N\\
	\mbox{satisfying } \phi \circ f_M = f_N \circ \phi
	\end{matrix}
	\right\}
	\]

The following theorem uses the equivariance of the glue map to impose compatibility conditions on the minimal polynomials of the two actions.
\begin{theorem}\label{thm:glue_structure}
Let $M \oplus N \hookrightarrow L$ be a primitive extension and $f_M,f_N$ be endomorphisms of $M$ and $N$ with minimal polynomials $m(x)$ and $n(x)$. Suppose that $f_M \oplus f_N$ extends to $f:L\rightarrow L$.
Then \[d L \subseteq M \oplus N\] where $d \ZZ = (m(x) \ZZ[x] + n(x) \ZZ[x]) \cap \ZZ$.
\end{theorem}
\begin{proof}
  If $m(x)$ and $n(x)$ are not coprime in $\QQ[x]$, then $d=0$, and the theorem holds.
  Assume that they are coprime in $\QQ[x]$. 
  By the definition of $d$ we can find polynomials $u,v \in \ZZ[x]$ such that
  \[d =u(x)n(x)+v(x)m(x).\]
  Then, inserting $f$ for $x$, we obtain $d \cdot \id_L=u(f)n(f)+v(f)m(f)$ as endomorphism of L.
  In the following we estimate the image of this endomorphism. Recall that $fL \defeq f(L)$ denotes the image of $f$. We note that since $n(x)$ and $m(x)$ are coprime and $M$ is primitive in $L$, $n(f)L = \ker m(f)=M$ and similary for $N$. 
  \begin{eqnarray*}
  dL & = & (u(f)n(f)+v(f)m(f))L \\
      &\subseteq & u(f)n(f)L + v(f)m(f)L\\
      & \subseteq & \ker m(f) + \ker n(f) \\
      &= & M \oplus N.
  \end{eqnarray*}
\end{proof}

\begin{corollary}\label{coro:dr-elementary}
  If $L$ is $r$-elementary, i.e. $r D_L=0$, then $M$ is $dr$-elementary, i.e.,
  \[dr  D_M=0.\]
  In particular,
  \[\det M \mid (dr)^{\rk M}.\]
  	
\end{corollary}
\begin{proof}
	We take the chain of inclusions
	\[M\oplus N \subseteq L \subseteq L^\vee \subseteq M^\vee \oplus N^\vee\]
	From Theorem \ref{thm:glue_structure} we get that $d L \subseteq M \oplus N$, and projecting this orthogonally to $M^\vee$ gives
	\[d \; p_M(L)=p_M(dL)\subseteq p_M(M \oplus N) = M.\]
	Together with Lemma \ref{lem:structure_glue_quotient} this implies
	\[dr D_M \subseteq dG_M = d p_M(L)/M=0\]
	as desired.
\end{proof}

\begin{example}
 As an example consider a non-symplectic automorphism $f$ of a K3 surface of order $4$.
 Set $L = H^2(X,\mathbb{Z})$, $m(x) = x^2+1$ and $n(x) = x^2 -1$. Since $L$ is unimodular, $r=1$. One calculates $d = 2$, and thus that $M = \ker f^2+1$ is $2$-elementary.
\end{example}

Note that $d$ divides the resultant $\res(m(x),n(x))$ and both have the same prime factors. But usually $d<res(m(x),n(x))$. For example $\res(x^2+1,x^2-1)=4$. We deduce the following  corollary. It was originally stated in \cite[Theorem 4.3]{mcmullen:minimum} for {\em unimodular} primitive extensions. 

\begin{corollary}\label{thm:resultant}
 Let $M,N,L$ be lattices and 
 \[M\oplus N \hookrightarrow L\]
 a primitive extension with glue $ G_M\cong G \cong G_N$.
 Let $f_M,f_N$ be isometries of $M$ and $N$ with characteristic polynomials
 $\chi_M$ and $\chi_N$. If $f_M\oplus f_N$ extends to $L$, then any prime dividing $|G|$ also divides the resultant $\res(\chi_M,\chi_N)$.
\end{corollary}

Since we are concerned with isometries of finite order, we give formulas for $d$ and the resultants in the case of cyclotomic polynomials. 
\begin{theorem}\cite[Thm. 1]{apostol:resultant}\cite[Thm. 1]{dresden:resultant}\label{thm:cyclotomic_resultants}
 The resultant of two cyclotomic polynomials $c_m,c_n$ with $m < n$ is given by
  \begin{equation*}
    \res(c_n,c_m)=
    \begin{cases}
     p^{\varphi(m)} &\mbox{if } n/m=p^e \mbox{ is a prime power,}\\
     1 & \mbox{ otherwise.}
     \end{cases}
  \end{equation*}
\end{theorem}

\begin{theorem}\cite[Thm. 2]{dresden:resultant}\label{thm:cyclotomic_intersection}
	\[(\ZZ[x]c_n+\ZZ[x]c_m)\cap \ZZ= \begin{cases}
		p &\mbox{if } n/m=p^e \mbox{ is a prime power,}\\
		1 & \mbox{ otherwise.}
	\end{cases}\]
\end{theorem}

\subsection{Real orthogonal transformations and the sign invariant}
In this section we review how to classify conjugacy classes of real orthogonal tranformations.
The classification works in terms of the so called sign invariant. Proofs and details can be found in
\cite[§2]{gross:unramified}.\\

We denote by $\mathbb{R}^{p,q}$ the vector space $\mathbb{R}^{p+q}$ equipped with the quadratic form
\[x_1^2+\dots+x_p^2-x^2_{p+1}- \dots -x^2_{p+q}.\]
Let $SO_{p,q}(\mathbb{R})=SO(\mathbb{R}^{p,q})$ be the Lie group of real orthogonal transformations of determinant one, preserving the quadratic form. 
If the characteristic polynomial $\chi(x)$ of $F \in SO_{p,q}(\mathbb{R})$ is of even degree $2n=p+q$ and separable, then
it is reciprocal, i.e., $x^{2n}\chi(x)=\chi(x^{-1})$. It has a \emph{trace polynomial} $r(x)$ defined by \[\chi(x)=x^nr(x+x^{-1}).\] Its roots are real of the form $\lambda + \lambda^{-1}$ where $\lambda$ is a root of $\chi(x)$. Call $\mathcal{T}$ the set of roots of $r(x)$ in the interval $(-2,2)$. They correspond to conjugate pairs of roots $\lambda+\overline{\lambda}$ of $\chi(x)$ on the unit circle. 
We have an orthogonal direct sum decomposition
\[\mathbb{R}^{p,q}=\bigoplus_{\tau \in \mathbb{R}} E_\tau, \quad E_\tau\defeq \ker(F+F^{-1}-\tau I).\]
On $E_\tau, \tau \in \mathcal{T}$, $F$ acts by rotation by angle $\theta=\arccos(\tau/2)$. Hence $E_\tau$ is either positive or negative definite. For $\tau \in \mathcal{T}$ this is encoded in the \emph{sign invariant}. 
\[\epsilon_F(\tau)=\begin{cases}+1 & \mbox{ if } E_\tau \mbox{ has signature } (2,0),\\
				-1 & \mbox{ if } E_\tau \mbox{ has signature } (0,2).
		   \end{cases}
\]
Denote by $2t$ the number of roots of $\chi(x)$ outside the unit circle. We can recover the signature via
\[(p,q)=(t,t)+\sum_{\tau \in \mathcal{T}}\begin{cases}
                     (2,0) & \mbox{ if } \epsilon_F(\tau)=+1\\
                     (0,2) & \mbox{ if } \epsilon_F(\tau)=-1\\
                    \end{cases}
\]
Two isometries $F,G\in SO_{p,q}(\mathbb{R})$ with the same characteristic polynomial are conjugate in $O_{p,q}(\mathbb{R})$ iff $\epsilon_F=\epsilon_G$.

\subsection{Lattices in number fields}\label{subsection:lattices_in_number_fields}
In this section we review the theory of lattice isometries associated to certain reciprocal polynomials as exploited in
\cite[§5]{mcmullen:minimum}. For further reading consider \cite{bayer-fluckiger:lattices_and_number_fields,bayer-fluckiger:given_characteristic_polynomial,bayer-fluckiger:ideal_lattices}.

\begin{definition}
A pair $(L,f)$ where $L$ is a lattice and $f\in O(L)$ an isometry with characteristic polynomial 
$p(x)$, is called a \emph{$p(x)$-lattice}. 
\end{definition}
We call two $p(x)$-lattices $(L,f)$ and $(N,g)$ isomorphic if there is an isometry $\alpha: L \rightarrow N$ with $\alpha \circ f = g \circ \alpha$.  

\begin{example}
 Let $X$ be a complex K3 surface and $f$ an automorphism of $X$ acting by multiplication with an $n$-th root of unity on $H^0(X,\Omega_X^2)$. If $\rk T(X)=\varphi(n)$, then $(\Trans(X),f)$ is a $c_n(x)$-lattice, where $c_n(x)$ denotes the $n$-th cyclotomic polynomial. To see this, note that 
 \[H^{2,0}\subseteq \left(\ker c_n(f^*|T)\right)\otimes_\mathbb{Z} \mathbb{C} \subseteq T \otimes \mathbb{C}.\]
 Since the kernel is defined over $\mathbb{Z}$, the equality $T=\ker c_n(f|T)$ follows from the minimality of $T$.
\end{example}

If we start with an irreducible, reciprocal polynomial $p(x) \in \mathbb{Z}[x]$ of degree $d=2e$, we can associate a $p(x)$-lattice to it as follows. Recall that $r(y)$ denotes the associated trace polynomial defined by $p(x)=x^er(x+x^{-1})$. 
Then
\[K\defeq \mathbb{Q}[f]\cong \mathbb{Q}[x]/p(x)\]
is an extension of degree $2$ of 
\[k\defeq \mathbb{Q}[f+f^{-1}]\cong \mathbb{Q}[y]/r(y)\]
with Galois involution $\sigma$ defined by $f^\sigma=f^{-1}$. We denote by $\O_K$ (resp. $\O_k$) the rings of integers of $K$ (resp. $k$) and $\mbox{Tr}^K_\QQ$ denotes the trace of $K$ over $\QQ$ 

\begin{definition}
The \emph{principal $p(x)$-lattice} is $(L_0,f_0)$ with 
\[L_0\defeq \mathbb{Z}[x]/p(x), \qquad f_0: L_0 \rightarrow L_0 \quad h\mapsto x \cdot h\]
and inner product defined by
\[\langle g_1,g_2\rangle_0\defeq \mbox{Tr}_\mathbb{Q}^K\left(\frac{g_1g_2^\sigma}{r'(x+x^{-1})}\right).\]
\end{definition}
We note that this form is even with $\left|\det L_0\right|=\left|p(1)p(-1)\right|$.

From a given $p(x)$-lattice one can construct new ones:
\begin{definition}
For $a\in \mathbb{Z}[f+f^{-1}]\subseteq \End(L)$ define a new inner product
\[\langle g_1,g_2\rangle_a\defeq \langle a g_1, g_2\rangle\]
on $L$. We denote the resulting lattice by $L(a)$, and call this operation a \emph{twist}. 
The pair $(L(a),f)$ is called a twisted $p(x)$-lattice. 
\end{definition}
We note that if $L$ is even, then so is its twist $L(a)$ \cite[4.1]{mcmullen:entropy_and_glue}. 

The situation is particularly nice if $K$ has class number one, $\mathbb{Z}[x]/p(x)$ is the full ring of integers $\mathcal{O}_K$ of $K$ and $|p(1)p(-1)|$ is square-free. In this case $p(x)$ is called a \textit{simple} reciprocal polynomial and we get the following theorem.

\begin{theorem}\cite[5.2]{mcmullen:minimum}\label{thm:principal}
 Let $p(x)$ be a simple reciprocal polynomial, then every $p(x)$-lattice is isomorphic to a twist $(L_0(a),f_0)$ of the principal $p(x)$-lattice.  
\end{theorem}
\begin{remark}\label{rmk:2principal}
 If we drop the condition that $|p(1)p(-1)|$ is square-free, we have to allow twists in $r'(x+x^{-1})/\mathcal{D}_K\cap \mathcal{O}_k=1/(x-x^{-1})\mathcal{O}_k$, where $\mathcal{D}_K=(p'(x))\mathcal{O}_K$ is the different of $K$.
 If $K/k$ ramifies over $2$, these need not be even in general.
 Dropping the condition on the class number leads to so called ideal lattices surveyed in \cite{bayer-fluckiger:ideal_lattices}.
\end{remark}

If $\mathbb{Z}[f]\cong \mathbb{Z}[x]/p(x)$ is the full ring of integers $\mathcal{O}_K$, then the discriminant group, glue, (dual) lattice, etc. are $\mathcal{O}_K$-modules. 
\begin{lemma}\label{lem:discriminant}
 Let $p(x)$ be a simple reciprocal polynomial. Then there is an element $b \in \mathcal{O}_K$ of absolute norm $|p(1)p(-1)|$ such that $L^\vee_0=\tfrac{1}{b}\mathcal{O}_K$. If $a\in \mathcal{O}_k$ is a twist, then
 \[L_0(a)^\vee \!/L_0(a)\cong \mathcal{O}_K/ab\mathcal{O}_K\]
 as $\mathcal{O}_K$-modules. 
\end{lemma}
\begin{proof} 
 Since $L_0^\vee\subseteq K$ is a finitely generated $\mathcal{O}_K$-module, it is a fractional ideal. By simplicity of $p(x)$ $\mathcal{O}_K$ is a PID and fractional ideals are of the form $\frac{1}{b}\mathcal{O}_K$, for some $b\in \mathcal{O}_K$. Then $L_0(a)^\vee=\frac{1}{a}L_0^\vee=\frac{1}{ab}L_0$ and 
 $D_{L_0(a)}\cong \mathcal{O}_K/ab\mathcal{O}_K$.
\end{proof}

Given a unit $u \in \mathcal{O}_K^\times$ and $a \in \mathcal{O}_K\setminus \{0\}$ the twist $L_0(uu^\sigma a)$ is isomorphic to $L_0(a)$ via $x\mapsto ux$ as $p(x)$-lattice. Conversely, if $v\in \mathcal{O}_k$ and $L_0(va)\cong L_0(a)$ as $p(x)$-lattices, then, by non-degeneracy of the trace map, we can find $u\in \mathcal{O}_K$ with $v=uu^\sigma$.
Since the cokernel of the norm map $N: \mathcal{O}_K^\times \rightarrow \mathcal{O}_k^\times$ is finite,
the associates of $a \in \mathcal{O}_k$ give only finitely many non-isomorphic twists.\\

By Lemma \ref{lem:discriminant} the prime decomposition of $a\in \mathcal{O}_k$ in $\mathcal{O}_K$ determines the $\mathcal{O}_K$-module structure of the discriminant, while twisting by a unit may change the signature (and discriminant form).\\

Let $\mathcal{T}$ denote the set of real roots 
$\tau$ of $r(y)$ in the interval $(-2,2)$. Each root $\tau \in \mathcal{T}$ corresponds to a real place  of $k=\mathbb{Q}[y]/r(y)$ that becomes complex in $K$. It is given by the embedding $\nu_\tau : k\hookrightarrow \mathbb{R}$ defined by $y \mapsto \tau$.
The \emph{sign map} is the homomorphism
\[\sign\colon \O_k^\times \rightarrow \{\pm 1\}^\mathcal{T}\]
defined by $\sign(u)_\tau=\sign(\nu_\tau(u))$.
Using the sign map, we can compute the sign invariant $\epsilon_{f_0}$ of the isometry $f_0$ of a twist $(L_0(a),f_0)$, $a \in \O_k$ of the principal lattice
\begin{equation}\label{eqn:sign_map} 
\epsilon_{f_0}(\tau)=\sign((a/r'(y))
\end{equation}
(cf. \cite[4.2]{gross:unramified}). 

All in all we have reviewed the classification of $p(x)$-lattices for $p(x)$ a simple reciprocal polynomial.
In a concrete situation the $p(x)$-lattices (of given determinant) can be enumerated by a simple computer program.

\section{Small cyclotomic fields}
Motivated by the action of a non-symplectic automorphism on the transcendental lattice of a K3 surface, we study $c_n(x)$-lattices more closely. In order to do this we review some of the general theory on cyclotomic fields. 
Our main reference is \cite{washington:introduction_to_cyclotomic_fields}.\\

For $n\in \mathbb{N}$, we denote by $K=\mathbb{Q}(\zeta_n)$ the $n$-th cyclotomic field and by $c_n(x)$ the $n$-th cyclotomic polynomial. The Euler totient function $\varphi(n)$ records the degree of $c_n(x)$.

The maximal real subfield of $K$ is $k=\mathbb{Q}[\zeta_n+\overline{\zeta}_n]$.
The rings of integers of these two fields are
\[\mathcal{O}_K=\mathbb{Z}[\zeta_n] \quad \mbox{and} \quad \mathcal{O}_k=\mathbb{Z}[\zeta_n+\overline{\zeta}_n].\]
As before denote by $\mathcal{T}$ the set of real places of $k$ that become complex in $K$. 
We have that $|\mathcal{T}|= \varphi(n)/2$.
\begin{lemma}
 The cyclotomic polynomials $c_n(x)$ are simple reciprocal polynomials for $2 \leq \varphi(n) \leq 21$, $n\neq2^d$.
\end{lemma}
\begin{proof}
 The only non-trivial part is that the class numbers are one. This is stated in \cite[main thm.]{masely:class_number}.
\end{proof}
Note that even though $|c_{2^d}(1)c_{2^d}(-1)|=4$ is not square-free, every \emph{even} $c_{2^d}$-lattice $(2\leq d \leq 5)$ is a twist of the principal $c_{2^d}$-lattice (cf. Remark \ref{rmk:2principal}). 

\begin{lemma}\cite[Prop. 2.8]{washington:introduction_to_cyclotomic_fields}\label{lem:zeta}
 If $n\in \mathbb{N}$, has two distinct prime factors, then $(1-\zeta_n)$ is a unit in $\mathcal{O}_K$. 
\end{lemma}

The kernel $\O_k^{\times +}$ of the sign map
\[\sign: \mathcal{O}_k^\times \rightarrow \{ \pm 1 \}^{\mathcal{T}} \]
is the set of \emph{totally positive units}\index{Totally positive unit} of $\O_k$. 
\begin{proposition}\cite[A.2]{shimura:cm}\label{prop:relative_class_number}
If the relative class number $h^-(K)=h(K)/h(k)$
is odd, then $\O_k^{\times +}=N^K_k(\O_K^\times)$.
\end{proposition}
\begin{corollary}\label{coro:sign_invariant_injective}
 Let $n\in \mathbb{N}$ with $\varphi(n)\leq 20$, and $K\defeq \mathbb{Q}(\zeta_n)$ be the $n$-th cyclotomic field. Then the sign map
 \[\sign: \mathcal{O}_k^\times/N(\mathcal{O}_K) \rightarrow \{ \pm 1 \}^{\mathcal{T}} \]
 is injective. 
\end{corollary}
\begin{proof}
As $\QQ[\zeta_n]$ is a PID for $\varphi(n)\leq 20$, the relative class number is one and we may apply Proposition \ref{prop:relative_class_number}.

\end{proof}
The first cases where the relative class number is even is for $n=39,56,29$. There $h^-(\QQ[\zeta_{n}])=2,2,2^3$ (cf. \cite[§3]{washington:introduction_to_cyclotomic_fields}), and the sign map has a kernel of order $2,2,2^3$ as well. 

\begin{proposition}\label{prop:p-lattice_isomorphism_class}
 The isomorphism class of a $c_n(x)$-lattice $(L,f)$, with $2\leq \rk L=\varphi(n)\leq 20$ is given by the kernel of 
 \[\mathbb{Z}[x]/c_n(x)\rightarrow O(L^\vee \!/L), \quad x \mapsto f|L^\vee \!/L\]
 and the sign invariant of $f$.
\end{proposition}
\begin{proof}
 By Theorem \ref{thm:principal}, $(L,f)$ is isomorphic to a twist $(L_0(a),f_0)$ of the principal $c_n(x)$-lattice where $a\in \O_k$. Lemma \ref{lem:discriminant} shows that the $\mathcal{O}_K$-module structure of the discriminant determines the prime decomposition of $a$. Thus we can write $a=ub$ for some fixed $b$ with a unit $u \in O_k^\times$. Now, the isomorphism class of $(L_0(a),f_0)$ depends on $u$ only modulo the image of the norm map $N: \O_K^\times \rightarrow \O_k^\times$. But by Corollary \ref{coro:sign_invariant_injective} this is captured by the sign map $\sign(u)$ which in turn computes the sign invariant $\epsilon_{f_0}$ of $f_0$ by (\ref{eqn:sign_map}).
\end{proof}
 
\section{Transcendental cycles and uniqueness of $X$ up to isomorphism.}
In this section we prepare the proof of Theorem \ref{thm:main}.
\begin{proposition}\label{prop:main}
 Let $X_i$, $ i \in \{1,2\}$ be two complex K3 surfaces and $f_i \in\Aut(X_i)$ automorphisms with $f_i^*(\omega_i)=\zeta_n\omega_i$ on $\CC \omega_i = H^0(X_i,\Omega_{X_i}^2)$ where $\zeta_n=e^{2\pi i/n}$. Suppose that
 \begin{itemize}
  \item[(1)] $\rk \Trans(X_i)=\varphi(n)$ and
  \item[(2)] $\Trans(X_i) \hookrightarrow L_{K3}$ uniquely.
 \end{itemize}
 Then there exists an isomorphism $X_1 \cong X_2$ if and only if $I_1 = I_2$ where $I_i$ is the kernel of 
 \[\mathbb{Z}[x]/c_n(x) \rightarrow O(\Trans(X_i)^\vee \!/\Trans(X_i),\qquad x\mapsto f_i|\Trans(X_i)^\vee \! / \Trans(X_i).\] 
\end{proposition}
\begin{proof}
 Let $X_i$, $i\in \{1,2\}$ be two K3 surfaces and $f_i$ be automorphisms as in the theorem. Let $T_i = \Trans(X_i)$.
 Set $\tau=\zeta_n+\zeta_n^{-1}$ and $E_{\tau,i}=\ker(f_i|_{T_i}+f|_{T_i}^{-1}-\tau id_{T_i})$. Looking at $\omega_i,\overline{\omega_i} \in E_{\tau,i}\otimes \mathbb{C}$ with $\omega_i.\overline{\omega_i}>0$, we see that $E_{\tau,i}$ has signature $(2,0)$. Since the signature of $T_i$ is $(2,\varphi(n)-2)$, this determines the sign invariants of $(T_1,f_1)$ and $(T_2,f_2)$ which is recorded by the complex $n$-th root of unity $\zeta_n$.
 By assumption the discriminants of $(T_i,f_i)$ have the same $\mathcal{O}_K$-module structure and Proposition \ref{prop:p-lattice_isomorphism_class} implies that $(T_1,f_1)\cong(T_2,f_2)$ as $c_n(x)$-lattices. 

 Hence, we can find an isometry $\psi_T:T_1 \rightarrow T_2$ such that $f_2 \circ \psi_T=\psi_T \circ f_1$.
 The latter condition assures that $\psi_T$ is compatible with the eigenspaces of $f_{1}$ and $f_2$. Since $\rk T_i=\varphi(n)$, the eigenspaces of $f_i|T_i$ for $\zeta_n$ are $H^{2,0}(X_i)$. In particular, \[\psi_T(H^{2,0}(X_1))=H^{2,0}(X_2).\]
 Now, choose markings $\phi_i$ of $X_i$.
 They provide us with two embeddings $\phi_1$ and $\phi_2 \circ \psi$ of $T_1$ into $L_{K3}$. By assumption (2) any two embeddings are isomorphic. That is, we can find $\psi\in O(L_{K3})$ such that the following diagram commutes. 
\[ \begin{tikzcd}
  T_1 \arrow{r}{f_1} \arrow{d}{\psi_T} & T_1 \arrow{d}{\psi_T}\arrow[draw=none]{r}[sloped,auto=false]{\subseteq} 	&[-23pt]H^2(X_1,\mathbb{Z})\arrow{r}{\phi_1} &L_{K3}\arrow{d}{\exists \psi} \\
  T_2 \arrow{r}{f_2}       & T_2 \arrow[draw=none]{r}[sloped,auto=false]{\subseteq} 			        &H^2(X_2,\mathbb{Z})\arrow{r}{\phi_2} &L_{K3} 
 \end{tikzcd}\]
 By construction $\phi_2^{-1} \circ \psi \circ \phi_1$ is a Hodge isometry. By the weak Torelli Theorem $X_1$ and $X_2$ are isomorphic.
 Conversely, let $f_1$, $f_2 \in \Aut(X)$ with $f_1^*\omega=f_2^*\omega=\zeta_n \omega$. Note that $f_1\circ f_2^{-1}$ is symplectic. Then $(f_1\circ f_2^{-1})|T=id_T$ ,i.e., $f_1|T=f_2|T$ and in particular $I_1=I_2$.
\end{proof} 

\begin{remark}
 Replacing $f$ by a power $f^k$ with $k$ coprime to $n$, we can fix the action on the 2-forms. This corresponds to the Galois action $\zeta_n \mapsto \zeta_n^k$ on $\mathbb{Q}(\zeta_n)$. 
 In case the embedding is not unique, one can fix the isometry class of $\NS$. Then isomorphism classes of primitive embeddings with $T^\perp = \NS$ are given by glue maps $\phi:T^\vee \!/T \rightarrow \NS^\vee \!/\! \NS$ with
 $-q_T=q_{\NS} \circ \phi$ modulo the action of $O(\NS)$ on the left. We can also allow for an action of the centralizer of $f|T$ in $O(T)$ on the right.  
 \end{remark}

Let $n=\prod_i p_i^{e_i}$ be a natural number given in prime factorization.
We call $\rad(n)=\prod_i p_i$ the radical of $n$.
\begin{proposition}\label{prop:det}
 Let $f$ be a non-symplectic automorphism acting with order $n$ on the global 2-forms of a complex K3 surface $X$ with
 $\varphi(n)=\rk \Trans(X)$. Then 
 \[\det \Trans(X) \mid\res(c_n,\mu)\]
 where $\mu(x)$ is the minimal polynomial of $f|\NS(X)$.
 If $\res(c_n,\mu)\neq 0$ and $f$ is purely non-symplectic, then $\Trans(X)$ is $\rad(n)$-elementary, i.e. $\rad(n) D_{\Trans(X)}=0$.
 \end{proposition}
\begin{proof}  
  If $c_n$ and $\mu$ have a common factor, then $\res(c_n,\mu)=0$, and the statement is certainly true. 
  We may assume that $\gcd(c_n,\mu)=1$. Then we know that \[\Trans(X)=\ker c_n(f|H^2(X,\mathbb{Z}))\]
  and we can view $(\Trans(X),f)$ as a $c_n(x)$-lattice. Then $D_{\Trans(X)}\cong \mathcal{O}_K/I$, $K=\mathbb{Q}(\zeta_n)$, for some ideal $I<\mathcal{O}_K$. The isomorphisms $D_{\NS(X)}\cong D_{\Trans(X)} \cong \mathcal{O}_K/I$ are compatible with $f$. In particular $\mu(f|\NS(X))=0$ implies that
  $\mu(f|D_{\Trans(X)})=0$, i.e., $\mu(\zeta_n) \in I$. By definition of the norm $N$ and resultant
  \[|\det \Trans(X)|=|\mathcal{O}_K/I|=N(I) \mid N(\mu(\zeta_n))=\res(c_n,\mu).\]
 It remains to prove that $nD_{\Trans(X)}=0$. By Corollary \ref{coro:dr-elementary} it is enough to show that 
 \begin{equation}\label{eqn:intersection}
 	\rad(n) = \ZZ \subseteq \left(\ZZ[x]c_n+\ZZ[x]\mu\right)\cap \ZZ.
 \end{equation}
 Since $f$ is assumed to be purely non-symplectic, we may replace $\mu$ by $\prod_{d\mid n,d\neq n} c_d= \frac{x^n -1}{c_n(x)}$. The equation can now be checked with the help of a computer algebra system for any $n$ with $\varphi(n)\leq 20$. We used sageMath \cite{sage} for this.
\end{proof}

\begin{corollary}\label{coro:n_restrictions}
Suppose that $\rk \Trans(X)=\varphi(n)$ and $f$ is purely non-symplectic with $\gcd\left(\mu_{f|\NS(X)},c_n\right)=1$. Set $\mu = \mu_{f|\NS(X)}$.
Then we have the following restrictions on $\Trans(X)$:
 \begin{enumerate}
  \item $\Trans(X)$ has signature $(2,\varphi(n)-2)$;
  \item $2\leq \varphi(n)\leq 20$;
  \item $(x-1) \mid \mu$  and $\mu \mid \prod_{k < n}c_k$;
  \item$\deg \mu\leq 22-\varphi(n)$ and $\det \Trans(X) \mid\res(c_n,\mu)$;
  \item $\exists b \in \mathcal{O}_k$ such that $\Trans(X) \cong L_0(b)$ is a twist of the principal $c_n(x)$-lattice.
  \end{enumerate}
  The resulting determinants are listed in Table \ref{tbl:possible_dets}. 
\end{corollary}
\begin{proof}
 (1) and (2) are clear. (3) Since $f$ preserves an ample divisor $\mu$ is divided by $(x-1)$. For the second part note that $f$ is of order $n$ and $\mu$ not divided by $c_n$.
 (4) This is Proposition \ref{prop:det}. 
  (5) By assumption $T(X)$ is a $c_n(x)-$lattice of rank $\deg c_n$. For $2\leq \varphi(n)\leq 21$ all cyclotomic polynomials $c_n(x)$ are simple and Theorem \ref{thm:principal} provides the claim.
  It remains to compute the values of Table \ref{tbl:possible_dets}. This is easily done with the help of a computer algebra system. The author used SageMath \cite{sage} for this purpose.
 To illustrate the computation we do it for $n=28$.
 By Theorem \ref{thm:cyclotomic_resultants} a factor $c_k(x)$ of $\mu(x)$ will contribute to the resultant if and only if $n/k$ is a prime power. 
 Hence, the only possibilities are $c_4,c_7,c_{14}, c_4c_7, c_4c_{14}$ which are of degree $2,6,6,8,8$ and give resultants $7^2,2^6,2^6,2^6 7^2,2^6 7^2$. 
 The principal $c_{28}(x)$-lattice is unimodular. Hence $|\det \Trans(X)| = |\det L_0(b)| = |N^K_\QQ(b)|$.
 We investigate the prime factorization of $b$. 
 The unique prime dividing $2$ in $\O_k$ has norm $2^6$ when viewed in $\O_K$. Similiarly, the unique prime above $7$ is of norm $7^2$. These are the only possible prime factors of $b$ which results in the $4$ possible determinants
 $1,2^6,7^2$ or $2^6 7^2$. We can exclude $7^2$ and $2^6 7^2$ since a computation shows that there is no twist $b$ of the right signature $(2,10)$. This leaves us with determinants $1$ and $2^6$.  
\end{proof}

\begin{table}[ht]
\caption{Possible determinants of the transcendental lattice}\label{tbl:possible_dets}
\centering
{
\renewcommand{\arraystretch}{1.3}
\rowcolors{1}{}{lightgray}
\begin{tabular}[t]{lll}
 \toprule
  $n$ &$\varphi(n)$ &$\det T$ \\
 \hline
  $3,6  $	&$2 $&$3 $\\

  $4  $	&$2 $&$2^2 $\\

  $5,10  $&$4$&$5 $\\

  $7,14  $&$6 $&$7$\\
  
  $8  $&$4 $&$2^2,2^4$\\
  
  $9,18  $&$6 $&$3,3^3 $\\
  
  $11,22  $&$10 $&$11$\\
  
  $12  $&$4 $&$1,2^23^4,2^4$\\

  $13,26  $&$12 $&$13 $\\
  
  $15,30  $&$8 $&$5^2,3^4 $\\
  
  $16  $&$8 $&$2^2,2^4,2^6,2^8$\\

  $17,34  $&$16 $&$17 $\\

  $19,38  $&$18$&$19 $\\

  \bottomrule
  \end{tabular}
% \!\!\!\!
\rowcolors{1}{}{lightgray}
 \begin{tabular}[t]{lll}
 \toprule
  $n$ &$\varphi(n)$ &$\det T$ \\
 \hline 
  $20  $&$8 $&$2^4,2^4 5^2$\\
  
  $21,42  $&$12 $&$1,7^2$\\

  $24$ &$8$ &$2^2,2^6,2^2 3^4,2^6 3^4$\\
  
  $25,50  $&$20 $&$5$\\
  
  $27,54  $&$18 $&$3,3^3 $\\
  
  $28  $&$12 $&$1,2^6$\\
  
  $32  $&$16 $&$2^2,2^4,2^6$\\
  
  $33,66  $&$20 $&$1 $\\
  
  $36  $&$12 $&$1,3^4,2^6 3^2$\\
  
  $40  $&$16 $&$ 2^4 $\\
  
  $44  $&$20 $&$1$\\
  
  $48  $&$16 $&$2^2$\\
  
  $60  $&$16 $&$-$ \\
  \bottomrule
  \end{tabular}
}
\end{table}

\begin{lemma}\label{lem:unique_table}
 Let $2\leq \varphi(n) \leq 20$. 
 For each prime $p \mid n$, there is a unique prime ideal in $\mathcal{O}_k$ dividing $p$. 
 \end{lemma}
\begin{proof}
 Since we need the statement only for finitely many $n$, it can be checked with the help of a computer algebra system.
\end{proof}

\begin{lemma}
 The $c_n(x)$-lattices $(T,f)$ of Table \ref{tbl:possible_dets} are determined up to isomorphism 
 by their determinants and sign invariant. They admit a unique primitive embedding into $L_{k3}$, except $(n,\det T)=(32,2^6)$ which does not embed into $L_{K3}$.
\end{lemma}
\begin{proof}
 The $c_n$-lattices are twists of the principal $c_n$-lattice by elements in $O_k$.
 The determiant of $T$ gives the norm of the twist which by Lemma \label{lem:unique_prime_ideal} gives the prime factorization of the twist up to a unit. The unit is determined by the sign invariant (compare \ref{prop:p-lattice_isomorphism_class}) and thus the isomorphism class of $(T,f)$. 
 
 Now one can explicitly compute all $c_n(x)$-lattices in the table and check that $\varphi(n)+l(T^\vee \!/T) \leq 20$ for all pairs $(n,d)$ except $(25,5)$, $(27,3^3)$ and $(32,2^6)$. Then Theorem \ref{thm:main} provides uniqueness (and existence) of a primitive embedding outside those $3$ cases.
 
 We have to check in case $n=25$ that $T$ embeds uniquely into the K3-lattice. It has rank $20$ and determinant $5$.  Its orthogonal complement $\NS$ is an indefinite binary quadratic form of determinant $5$. It is unique in its genus and the canonical map $O(\NS)\rightarrow O(\NS^\vee \!/\NS)$ is surjective since both groups are generated by $-id$. By \cite[1.14.1]{nikulin:quadratic_forms} the embedding of $T$ into $L_{k3}$ is unique.
  
 For the case $(27,3^3)$, we need more theory not explained here, see e.g. \cite{miranda_morrison:embeddings}. 
 By \cite[VIII 7.6]{miranda_morrison:embeddings} $\NS$ is $3$-semiregular and $p$-regular for $p\neq 3$. Now \cite[VIII 7.5]{miranda_morrison:embeddings} provides surjectivity of 
 $O(\NS) \rightarrow O(\NS^\vee\!\!/\NS)$ and uniqueness in its genus (alternatively cf. \cite{miranda_morrison:embeddings_I,miranda_morrison:embeddings_II}). 
 Uniqueness of the embedding follows again with \cite[1.14.1]{nikulin:quadratic_forms}.
 
 It remains to check that $(32,2^6)$ does not embed into the K3-lattice. Suppose that it does.
 Then its orthogonal complement is isomorphic to $A_1(-1)\oplus 5A_1$ which is the only lattice of signature $(1,5)$ and discriminant group $\mathbb{F}_2^6$ (cf. \cite[Tbl. 15.5]{conway_sloane:sphere_packings}). Its discriminant form takes half integral values. Up to sign it is isomorphic to the discriminant form of its orthogonal complement $T\cong U(2) \oplus U(2) \oplus D_4 \oplus E_8$ which takes integral values, contradicting the existence of a primitive embedding. 
\end{proof}

\begin{proposition}\label{prop:transcendental_lattice_unique_embedding}
 Let $X$ be a complex K3 surface and $f \in\Aut(X)$ be an automorphism of finite order with $f^*\omega=\zeta_n\omega$ on $\CC \omega = H^0(X,\Omega_X^2)$. Suppose that $\rk T_X=\varphi(n)$. Then there is a unique primitive embedding
 \[T_X \hookrightarrow L_{K3}.\]
\end{proposition}
\begin{proof}
 If $\varphi(n)\leq 10$, then $\rk T_X + l(D_{\Trans(X)})\leq 2 \rk T_X=2\varphi(n)\leq 20$ and Theorem \ref{thm:unique embedding} provides uniqueness of the embedding.
 If $\varphi(n)>10$, then $\varphi(n)\geq 12$ and $\zeta_n$ is not an eigenvalue of $f|\NS$. 
 By Corollary \ref{coro:n_restrictions} there are only finitely many possibilities of $T_X$ up to isometry. Uniqueness of the embedding is checked individually in Lemma \ref{lem:unique_table}.
 \end{proof}

\begin{proof}[Proof of Theorem \ref{thm:main}]
The difference between Theorem \ref{thm:main} and Proposition \ref{prop:main} is that the automorphism $f_1$ is of finite order and we do not require the uniqueness of $T_i \hookrightarrow L_{k3}$.
Since $f_1$ is of finite order $T(X_1)$ has a unique embedding into the K3 lattice by Proposition \ref{prop:transcendental_lattice_unique_embedding}. Since $I_1=I_2$ and the sign invariants agree, we get $\Trans(X_1)\cong T(X_2)$ and thus uniqueness of the embedding for $\Trans(X_2)$ as well.
\end{proof}

\begin{lemma}\label{lem:54}
 The pair $(54,3^3)$ is not realized by a K3 surface. 
\end{lemma}
\begin{proof}
Suppose that it is realized. A computation of the $c_{54}(x)$-lattice of determinant $3^3$ shows that $D_T \cong \FF_3^3$ and that the action of $f$ on $D_T$ is of order $6$. 
Hence, $\NS \cong U(3)\oplus A_2$ and the action on $D_{NS}\cong D_T$ is of order $6$ as well. Using Vinberg's algorithm (see Lemma \ref{lem:vinberg}) one explicitly computes the exact sequence
\[1 \rightarrow W(\NS)\rightarrow O^+(\NS)\rightarrow O(q_{\NS}) \rightarrow\{\pm1\}\rightarrow 1\]
and the image of $O^+(\NS)$ in $O(q_{\NS})$. However, an enumeration of its $24$ elements shows that it does not contain an element of order six. But $f|\NS$ is an element $O^+(\NS)$, giving the contradiction.
\end{proof}

\begin{table}[ht]
 \caption{Realized determinants in ascending order of $\varphi(n)\leq 10$}\label{tbl:real_dets1}
\centering
{
\renewcommand{\arraystretch}{1.2}
\rowcolors{1}{}{lightgray}
\begin{tabular}{llllr}
 \toprule
  $n$ &$\det T$&$X$&$f$ & \\
  \hline
  $3,6$	&$3 $		&$y^2=x^3-t^5(t-1)^5(t+1)^2 $ &$(\zeta_3x,\pm y,t) $&\cite[(7.9)]{kondo:trivial_action}\\
   
  $4  $	&$2^2 $ 	&$y^2=x^3+3t^4x+t^5(t^2-1) $ &$ (-x,\zeta_4y,-t)$ &\\
   
  $5,10$&$5 $		&$y^2=x^3+t^3x+t^7 $	&$(\zeta_5^3x,\pm\zeta_5^2y,\zeta_5^2t)$ &\cite[(7.6)]{kondo:trivial_action}\\

  $8  $	&$2^2$ 		&$y^2=x^3+tx^2+t^7$ 	&$(\zeta_8^6x,\zeta_8y,\zeta_8^6) $&\\
	&$2^4$		&$t^4=(x_0^2-x_1^2)(x_0^2+x_1^2+x_2^2)$	&$(\zeta_8 t;x_1:x_0:x_2)$	&\\
   
     $12  $&$1$		&$y^2=x^3+t^5(t^2-1)$	&$(-\zeta_3x,\zeta_4 y,-t)$&\cite[(3.4)]{kondo:trivial_action}\\
   	&$2^23^2$	&$y^2=x^3+t^5(t^2-1)^2$ &$(-\zeta_3x,\zeta_4y,t)$&\\
   	&$2^4$		&$y^2=x^3+t^5(t^2-1)^3$	&$(-\zeta_{3}x,\zeta_{4}y,-t)$&\\

  $7,14$&$7$		&$y^2=x^3+t^3x+t^8 $	&$(\zeta_7^3x,\pm \zeta_7y,\zeta_7^2t) $&\cite[(7.5)]{kondo:trivial_action}\\
   
  $9,18$&$3$		&$y^2=x^3+t^5(t^3-1) $	&$(\zeta_9^2x,\pm \zeta_9^3 y,\zeta_9^3t) $&\cite[(7.8)]{kondo:trivial_action}\\
	&$3^3$		&$y^2=x^3+t^5(t^3-1)^2$	&$(\zeta_9^2 x,\pm y,\zeta_9^3t) $&\\      
   
  $16  $&$2^2$		&$y^2=x^3+t^2x+t^7$	&$(\zeta_{16}^2x,\zeta_{16}^{11}y,\zeta_{16}^{10}t)$&\cite[4.2]{tabbaa_sarti_taki:order16}\\
	&$2^4$		&$y^2=x^3+t^3(t^4-1)x$	&$(\zeta_{16}^6x,\zeta_{16}^9y,\zeta_{16}^4t)$&\cite[4.1]{dillies:order16}\\
	&$2^6$		&$y^2=x^3+x+t^8$	&$(-x,iy,\zeta_{16}t)$&\cite[2.2]{tabbaa_sarti_taki:order16}\\
      
  $20  $&$2^4$		&$y^2=x^3+(t^5-1)x$	&$(-x,\zeta_4y,\zeta_5t)$&\\
	&$2^45^2$		&$y^2=x^3+4t^2(t^5+1)x$	&$(-x,\zeta_4y,\zeta_5t)$&\\
   
  $24$ 	&$2^2$		&$y^2=x^3+t^5(t^4+1) $	&$(\zeta_3\zeta_8^6x,\zeta_8y,\zeta_8^2t)$&\\
	&$2^6$		&$y^2=x^3+(t^8+1) $	&$(\zeta_3 x,y,\zeta_8 t) $&\\
	&$2^2 3^4$	&$y^2=x^3+t^3(t^4+1)^2$	&$(\zeta_3\zeta_8^{6}x,\zeta_8y,\zeta_8^6t) $&\\
	&$2^6 3^4$	&$y^2=x^3+x+t^{12}$	& $(-x,\zeta_{24}^6y,\zeta_{24}t)$&\\

  $15,30$&$5^2$		&$y^2=x^3+4t^5(t^5+1)$	&$(\zeta_3 x,\pm y,\zeta_5 t)$&\\
	 &$3^4$	&$y^2=x^3+t^5x+1$	&$(\zeta_{15}^{10} x,\pm y,\zeta_{15}t)$&\\

  $11,22$&$11$		&$y^2=x^3+t^5x+t^2$	&$(\zeta_{11}^5x,\pm \zeta_{11}^2y,\zeta_{11}^2t) $&\cite[(7.4)]{kondo:trivial_action}\\
   
  \bottomrule
  \end{tabular}
}
\end{table}

\begin{table}[ht]
 \caption{Purely non-symplectic automorphisms with $\varphi(n)\geq 12$}\label{tbl:real_dets2}
\centering
{
\renewcommand{\arraystretch}{1.2}
\rowcolors{1}{}{lightgray}
\begin{tabular}{llllr}
 \toprule
  $n$ &$\det T$&$X$&$f$ & \\
  \hline
 
  $13,26$&$13 $		&$y^2=x^3+t^5x+t$	&$(\zeta_{13}^5x,\pm\zeta_{13}y,\zeta_{13}^2t) $&\cite[(7.3)]{kondo:trivial_action}\\
  $26$&$13 $		&$y^2=x^3+t^7x+t^4$	&$(\zeta_{13}^{10}x,-\zeta_{13}^2y,\zeta_{13}t)$ &\\
  \textcolor{red}{$26$}&\textcolor{red}{$13$} 		&\textcolor{red}{$w^2=x_0^4 y_0^4 + x_1^4 y_1^3 y_2 + x_0 x_1^3 y_1^4$}	&\textcolor{red}{$((\zeta_{13}x_0:x_1),(\zeta_{13}^9y_0:y_2),-\zeta_{13}^7 w)$} &\\
  $21,42$    &$1$		&$y^2=x^3+t^5(t^7-1) $	&$(\zeta_{42}^2x,\zeta_{42}^3y,\zeta_{42}^{18}t) $&\cite[(3.0.2)]{kondo:trivial_action}\\
  $21,42$	 &$7^2$		&$y^2=x^3+4t^4(t^7-1)$	&$(\zeta_3 \zeta_7^6 x,\pm \zeta_7^2 y,\zeta_7 t)$&\\
  $21,42$	 &$7^2$		&$y^2=x^3+t^3(t^7+1)$   &$(\zeta_3 \zeta_7^3 x,\pm \zeta_7 y,\zeta_7^3t)$ &\\
  $21$       &$7^2$ 	&$x_0^3x_1 + x_1^3x_2 + x_0x_2^3 - x_0x_3^3$& $(\zeta_7x_0:\zeta_7x_1:x_2:\zeta_3 x_3)$ &\\
   
  $28  $&$1$		&$y^2=x^3+x+t^7 $	&$(-x,\zeta_4 y,-\zeta_7t) $&\cite[(3.3)]{kondo:trivial_action}\\
	&$2^6$		&$y^2=x^3+(t^7+1)x $	&$(-x,\zeta_4 y,\zeta_7t) $&\\
   	&$2^6$		&$y^2=x^3+(t^7+1)x $	&\multicolumn{2}{l}{$\left(x-(y/x)^2,\zeta_4 \left(y-(y/x)^3\right),\zeta_7t\right)$}\\
  
  $17,34$&$17 $		&$y^2=x^3+t^7x+t^2 $	&$(\zeta_{17}^7x,\pm \zeta_{17}y,\zeta_{17}^2t)$&\cite[(7.2)]{kondo:trivial_action}\\
  $34$   &$17$		&$x_0x_1^5+x_0^5x_2+x_1^2x_2^4=y^2$ &$(-y;x_0:\zeta_{17}x_1,\zeta_{17}^5 x_2)$ &\\
  
  $32  $&$2^2$		&$y^2=x^3+t^2x+t^{11}$	&$(\zeta_{32}^{18}x,\zeta_{32}^{11}y,\zeta_{32}^2t)$  &\cite[2 (1)]{oguiso:order32_example}\\
        &$2^4$      &$y^2=x_0(x_1^5+x_0^4 x_2+x_1 x_2^4)$ & $(\zeta_{32}y;\zeta^{2}_{32}x_0:x_1:\zeta^{24}_{32}x_2)$ & \\

  $36  $&$1$		&$y^2=x^3-t^5(t^6-1) $	&$(\zeta_{36}^2x,\zeta_{36}^3y,\zeta_{36}^{30} t) $&\cite[(3.2)]{kondo:trivial_action}\\
 	&$3^4$		&$y^2=x^3+x+t^9$	&$(-x,\zeta_4y,-\zeta_9t) $&\\
 	&$2^6 3^2$	&$x_0x_3^3+x_0^3x_1+x_1^4+x_2^4$ &$(x_0:\zeta_9^3x_1:\zeta_4\zeta_9^3x_2:\zeta_9x_3)$&\\

  $40  $&$ 2^4 $	&$z^2=x_0(x_0^4x_2+x_1^5-x_2^5)$	&$(x_0:\zeta_{20}x_1:\zeta_4x_2;\zeta_8z) $&\cite[4 (15)]{oguiso_machida}\\

  $48  $&$2^2$		&$y^2=x^3+t(t^8-1)$	&$(\zeta_{48}^2x,\zeta_{48}^3y,\zeta_{48}^6 t) $&\cite[4 (5)]{oguiso_machida}\\
    
  $19,38$&$19 $		&$y^2=x^3+t^7x+t$	&$(\zeta_{19}^7t,\pm\zeta_{19}y,\zeta_{19}^2t) $&\cite[(7.1)]{kondo:trivial_action}\\
  $38$& $19$		&$y^2=x_0^5x_1+x_0x_1^4x_2+x_2^6$	&$(x_0:\zeta_{19}x_1:\zeta_{19}^{16}x_2;-\zeta_{19}^{10}y)$& \\
  $27,54$&$3$		&$y^2=x^3+t(t^9-1)$	&$(\zeta_{27}^2x,\zeta_{27}^3y,\zeta_{27}^6t) $&\cite[(7.7)]{kondo:trivial_action}\\
  $27$	 &$3^3 $	&$x_0x_3^3+x_0^3x_1+x_2(x_1^3-x_2^3)$ &$(x_0:\zeta_{27}^3x_1:\zeta_{27}^{21}x_2:\zeta_{27}x_3)$&\\
   
  $25,50$&$5$		&$z^2=(x_0^6+x_0x_1^5+x_1x_2^5)$&$(z;x_0:\zeta_{25}^5x_1:\zeta_{25}^4x_2)$&\cite[(7.12)]{kondo:trivial_action}\\
      
  $33,66  $&$1 $	&$y^2=x^3+t(t^{11}-1) $	&$(\zeta_{66}^2x,\zeta_{66}^3y,\zeta_{66}^6t) $&\cite[(3.0.1)]{kondo:trivial_action}\\
  
  $44  $&$1$		&$y^2=x^3+x+t^{11} $	&$(-x,\zeta_4 y,\zeta_{11}, t) $&\cite[(3.1)]{kondo:trivial_action}\\
   
  \bottomrule
  \end{tabular}
}
\end{table}
The following is a generalization of Vorontsov's theorem \cite{vorontsov}.
\begin{theorem}\label{thm:n_d}
 Let $X$ be a K3 surface and $f$ a purely non-symplectic automorphism of order $n$ such that $\rk T=\varphi(n)$ and $\zeta_n$ is not an eigenvalue of $f|\NS\otimes\CC$.
 
 Set $d=\left|\det \NS\right|$, then $X$ is determined up to isomorphism by the pair $(n,d)$.
 Conversely, all possible pairs $(n,d)$ and equations for $X$ and (some) $f$ are given in Tables \ref{tbl:real_dets1}, \ref{tbl:real_dets2}. 
\end{theorem}
\begin{proof}
 Comparing Tables \ref{tbl:possible_dets} and \ref{tbl:real_dets1}, \ref{tbl:real_dets2} we have to exclude the pairs $(16,2^8)$ and $(54,3^3)$.
 For $(16,2^8)$, this is done in \cite[4.1]{tabbaa_sarti_taki:order16}.
 The pair $(54,3^3)$ is ruled out in Lemma \ref{lem:54}.\\ 
 
  By Lemma \ref{lem:unique_table} the transcendental lattice is uniquely determined by $(n,d)$ and embeds uniquely into $L_{K3}$. By Theorem \ref{thm:main} $X$ is determined up to isomorphism by $(\zeta_n,I)$, where $I$ is the kernel of 
 \[\mathbb{Z}[x]/c_n(x) \rightarrow O(T^\vee \!/T),\qquad x\mapsto f|T^\vee \! /T\]
 and $f^*\omega_X=\zeta_n \omega_X$ for $\CC \omega_X = H^0(X,\Omega_X^2)$.
 By Lemma \ref{lem:unique_table}, $I$ is determined uniquely by $(n,d)$.
 Replacing $f$ with $f^k$, $(n,k)=1$, does not affect $(n,d)$, hence $I$. However, in this way we can fix a primitive $n$-th root of unity $\zeta_n$.\\
 
 It remains to compute the Néron-Severi group of the examples in Tables \ref{tbl:real_dets1}, \ref{tbl:real_dets2} not found in the literature. In most cases this can be done by collecting singular fibers of an elliptic fibration (see \cite[p. 365]{silverman:advanced_topics}) or determining the \emph{fixed lattice}\index{Fixed lattice} $S(f^k)=H^2(X,\ZZ)^{f^k}$ of a suitable power of the automorphism $f$ via its fixed points. The corresponding tables of fixed lattices are collected in \cite{artebani_sarti_taki:non-symplectic,artebani_sarti:order3}, \cite[Figure 2]{nikulin:lobachevsky}.\\

\boldmath$(4,2^2)$ \unboldmath We see two fibers of type $II^*$ over $t=0,\infty$ and two fibers of type $I_2$ over $t=\pm 1$. Then $\NS \cong U\oplus 2 E_8 \oplus 2 A_1$ as expected. The two form is given in local coordinates by $dx \wedge dt/2y$, and $f^*(dx \wedge dt/2y)= -dx \wedge -dt/(2y\zeta_4)=\zeta_4^3dx \wedge dt/2y$. Hence the action is non-symplectic. The fixed lattice is $U\oplus 2 E_8$ while the $I_2$ fibers are exchanged. Giving that $f|\NS$ has order two.\\

\boldmath$(8,2^2)$\unboldmath This belongs to the second family of \cite[Thm 1]{schuett:2power}
 with $m=4$ and $\Trans \cong U \oplus U(2)$ for $\lambda = 0$.\\

\boldmath$(8,2^4)$ \unboldmath  The fourfold cover of $\mathbb{P}^2$ is a special member of a family in \cite[Ex. 5.3]{artebani_sarti:order4}. It has five $A_3$ singularities. The fixed locus of the non-symplectic involution $f^4$ consists of $8$ rational curves, where each $A_3$ configuration contains 1 fixed curve. Hence, its fixed lattice is of rank $18$ and determinant $-2^4$. It equals $\NS$.\\
 
\boldmath $(12,2^4)$ \unboldmath We get fibers of type $1 \times II, 1 \times II^*, 2 \times I_0^*$ which results in the lattice $\NS=U \oplus E_8 \oplus 2 D_4$.\\
 
\boldmath $(12,2^2 3^2)$ \unboldmath This time the zero section and the fibers span the lattice $\NS=U\oplus E_8 \oplus 2 A_2 \oplus D_4$.\\
 
\boldmath$(15,5^2)$ \unboldmath This elliptic K3 surface arises as a degree $5$ base change from the rational surface $Y:y^2=x^3+4t(t+1)$.
We see the section $(x,y)=(\zeta_3^k,1+2t)$ of $Y$ and then $(x,y)=(\zeta_3^k,1+2t^5)$ generating the Mordell-Weil group of $X_{(15,5^2)}$. Alternatively one can compute that $\NS$ is the fixed lattice of $f^3$.\\

\boldmath$(15,3^4)$ \unboldmath The 5-th power $f^5$ of $f$ is a non-symplectic automorphism of order $3$ acting trivially on $\NS$.
 It has $2$ fixed curves of genus $0$ lying in the $E_7$ fiber and $6$ isolated fixed points over $t=0$ and $t=\infty$.
 The classification of the fixed lattices of non-symplectic automorphisms of order 3 provides the fixed lattice of $f^5$ which equals $\NS$. In order to get explicit generators of the Mordell-Weil group we can base change with $t\mapsto t^3$ from the rational surface $y^2=2x^3+tx+t^4$ with sections $(x,y)=(t,t^2+t),(0,t^2)$.\\
 
 \textcolor{red}{
 \boldmath$(26,13)$ \unboldmath
 The third case is the minimal resolution of a double cover of $\PP^1 \times \PP^1$ branched over the curve $x_0^4 y_0^4 + x_1^4 y_1^3 y_2 + x_0 x_1^3 y_1^4$ of bi-degree $(4,4)$.
 The double cover has an $E_6$ singular point above $((1:0),(0:1))$. The branch curve is fixed by the covering involution. Its strict transform is a smooth curve of genus $6$. The involution fixes a single smooth rational curve coming from the resolution of the $E_6$ singularity.
 Therefore the fixed lattice of the involution has rank $6$ and determinant $2^4$.}

 \boldmath$(34,17)$ \unboldmath 
 In the first case the fixed locus of $g_1^{17}$ consists of a curve ($y=0$) of genus $8$ and a rational curve - the zero section. This leads to the fixed lattice $U\oplus (-2) \oplus (-2)$.

 Since the fixed locus of $g_2^{17}$ is a curve of genus $8$, $S(g^{17})\cong(2)\oplus (-2) \oplus (-2)$. Note that there is an $A_4$ singularity at zero.
 Since the fixed lattices of the two automorphisms are different, the actions are distinct as well.\\
 
 \boldmath$(20,2^4)$ \unboldmath The elliptic fibration has $5$ fibers of type $III$ and a single fiber of type $III^*$. This results in the lattice $U\oplus E_7 \oplus 5 A_1$ spanned by fiber components and the zero section. It has determinant $2^6$. 
 Since there is also a $2$-torsion section, $\det\NS=2^4$.\\

 \boldmath$(20,2^4 5^2)$ \unboldmath In this case $X_{(20,2^4 5^2)}$ has a single fiber of type $I_0^*$ and $6$ fibers of type $III$. This results in the lattice $U \oplus D_4 \oplus 6 A_1$ of rank $12$ and determinant $2^8$. Again there is $2$-torsion. We reach a lattice of determinant $2^6$. We get $X$ by a degree $5$ base change from $y^2=x^3+4t^2(t+1)x$
 with section $(x,y)=(t^2,t^3+2t^2)$. We get the sections $(x,y)=(t^6,t^9+2t^4)$ and $(-x(t),iy(t))$ generating the Mordell-Weil lattice $A_1^\vee(5)^{\oplus 2}$ of $X$. \\
 
 \boldmath$(21,7^2)$ \unboldmath We can base change this elliptic fibration from $y^2=x^3+4t^4(t-1)$ to get the sections
 $(x,y)=(\zeta_3^k t^6,2t^2+t^9)$ which generate the Mordell-Weil lattice of $X$. Using the height pairing one can then compute $\NS$. Alternatively note that $f^3$ is an order $7$ non-symplectic automorphism acting trivially on $\NS$ and not fixing a curve of genus $0$ point-wise. There is only a single possible fixed lattice of rank $10$, namely $U(7)\oplus E_8$.
 For the other possible actions see Lemma \ref{lem:21,7^2}.\\
 
\boldmath $(24,2^2)$ \unboldmath The fibration has $4$ type $II$ fibers, one type $I_0^*$ and an $II^*$ fiber. We get $\NS=U\oplus D_4 \oplus E_8$.\\
 
\boldmath $(24,2^6)$ \unboldmath In this case the fixed locus of $f^{12}$ consists of $4$ rational curves and a curve of genus $1$. This leads to a fixed lattice of rank $14$ and determinant $2^6$ as expected.\\
 
\boldmath$(24,2^2 3^4)$ \unboldmath The trivial lattice is $U\oplus D_4 \oplus 4 A_2$. It equals $\NS$ for absence of torsion sections. \\

 \boldmath$(24,2^63^4)$ \unboldmath Since $X$ has a purely non-symplectic automorphism of order $24$, the rank of $\NS$ is either $6$ or $14$.  Since Fix($f^{12}$) consists of $2$ smooth curves of genus $0$, its fixed lattice $S(f^{12})$ is $2$-elementary of rank $10$ and determinant $-2^8$. 
 Hence, we see that $\rk \NS=14$. Since the orthogonal complement of $S=S(f^{12})$ in $\NS$ is of rank $4$, the glue $G_S$ is an at most $4$ dimensional subspace. Then $2^4\leq |D_S/G_S|\leq |D_{\NS}|$ by Lemma \ref{lem:glue estimate}. Hence $2^4 \mid \det \NS$. Note that $S(f^{12})=\ker {f^{12}-1}$ and then $\ker c_{24}c_8(f)=S(f^{12})^\perp$ is of rank $12$. This shows that the characteristic polynomial of $f|\NS$ is divided by $c_8$ but not by $c_{24}$.
We are in the situation of Theorem \ref{thm:n_d}. As $2^4 \mid \det \NS$, it is either $-2^6$ or $-2^63^4$. We show that $3\mid \det \NS$. 
A direct computation reveals that Fix$(f^8)$ consists of a smooth curve of genus $1$ and $3$ isolated fixed points. This leads to a fixed lattice 
\[S(f^8)=\ker(c_8c_4c_2c_1)(f)\cong U(3)\oplus 3 A_2\]
of rank $8$ and determinant $-3^5$.  Now we view $S(f^8)$ as a primitive extension of $\ker c_8(f) \oplus \ker c_4c_2c_1(f)$. The rank of both summands is $4$, while the length of the discriminant group of $S(f^8)$ is $5$. 
Then each summand must contribute to the discriminant group. We see that $3 \mid \det \ker c_8(f)$. However, 
\[3 \nmid\res(c_8,c_{12}c_{6}c_{4}c_3c_2c_1)=2^4.\] 
In particular the $3$ part of $D_{\ker c_8(f)}$ cannot be glued inside $\NS$. Then $3 \mid \det \NS$.  \\
 
\boldmath $(27,3^3)$ \unboldmath The action of $f^9$ has an isolated fixed point and a fixed curve of genus $3$. We see that the fixed lattice of $f^9$ is $U(3)\oplus A_2=\NS$. It is spanned by the $4$ lines at $x_3=0$. Note that $f^3$ acts trivially on $\NS$ while $f$ does not.\\
 
\boldmath $(28,2^6)$ \unboldmath This fibration has $8$ fibers of type $III$ and a $2$-torsion section. Together they generate the Néron-Severi group.\\

\boldmath $(32,2^2)$ \unboldmath
The elliptic fibration has a singular fiber of type $I_0^*$, of type $II$ and $16$ of type $I_1$. Thus $\NS\cong U\oplus D_4$. Here $f$ has $6$ isolated fixed points.\\

\boldmath $(32,2^4)$ \unboldmath
 The fixed locus of $S(f^{16})$ is the strict transform of $y=0$ which is the disjoint union of a rational curve and a curve of genus $5$. Thus $\det NS=2^4$.
Note that $f$ has $4$ isolated fixed points.\\

\boldmath $(36,3^4)$ \unboldmath The fixed curves of $f^{12}$ are a smooth of genus $0$ over $t=0$ and the central rational curve in the $D_4$ fiber. This leads to the fixed lattice $U\oplus 4A_2=\NS$.\\
 
 \boldmath$(36,2^6 3^2)$ \unboldmath If we can show that $2\mid \det\NS$, then the only possibility is $\det\NS=-2^6 3^2$.
 The action of $f^{18}$ fixes a smooth curve of genus $3$ (it is a plane curve of degree $4$ in the hypersurface defined by $x_2=0$) and nothing else. Hence, its fixed lattice $S$ is $2$-elementary of rank $8$ and determinant $2^8$. Denote by $C=S^\perp \subseteq\NS$ the orthogonal complement of $S$ inside $\NS$. It has rank $2$.
 Assume that $2 \nmid \det\NS$. Then $ (D_S)_2 \cong (D_C)_2$ which is impossible, since $(D_S)_2$ has dimension $8$, while $(D_C)_2$ is generated by at most $2$ elements.
\end{proof}

\begin{remark}
The pair $(21,7^2)$ contradicts \cite[2.1]{jang:order21}. There it is claimed that a purely non-symplectic automorphism of order 21 acts trivially on $\NS$. As a consequence it is claimed that there is only a single K3 surface of order $21$. However there are two. The pair $(28,2^6)$ and its uniqueness are probably known to J. Jang independently. \\
The pair $(32,2^4)$ contradicts \cite[Thm 1.2]{taki:order32}.
There the uniqueness of $(X,\langle g \rangle)$ where $g$ is a purely non-symplectic automorphism of order $32$ is claimed.

 In \cite[1.8, 4.8]{taki:3power} non symplectic automorphisms of $3$-power order acting trivially on $\NS$ are classified. The author is missing a case. It is claimed that if $\NS(X)=U(3)\oplus A_2$ then there is no purely non-symplectic automorphism of order $9$ acting trivially on $\NS$. The pair $(27,3^3)$ contradicts this result - there the automorphism 
 acts with order $3$ on $\NS$. It is a special member of the family
 \[x_0x_3^3+x_0^3x_1+x_2(x_1-x_2)(x_1-ax_2)(x_1-bx_2)\]
 with automorphism given by $(x_0:x_1:x_2:x_3)\mapsto (x_0:\zeta_9^3x_1:\zeta_9^3x_2:\zeta_9x_3)$ and generically trivial action on $\NS$ as a fixed point argument shows. It was found by first computing the action of $f$ on cohomology through gluing, thus proving its existence. Then one specializes a family with automorphism of order $3$ given in \cite[4.9]{artebani_sarti:order3}.
\end{remark}

\section{Classification of non-symplectic automorphisms of high order}
Let $X_i$, $i=1,2$, be K3 surfaces and $G_i \subset \Aut(X_i)$ subsets of their automorphism groups. Recall that the pairs $(X_1,G_1)$, $(X_2,G_2)$ are said to be equivalent if there is
an isomorphism $F: X_1 \rightarrow X_2$ with $G_2 = FG_1F^{-1} \defeq \{F \circ g F^{-1}| g \in G_1\}$.
\begin{theorem}\label{thm:pair_finite_aut}
 Let $X_{(n,d)}$ be as in Tables \ref{tbl:real_dets1},\ref{tbl:real_dets2}.
\begin{enumerate}
	\item For $(n,d)=(66,1),(48,2^2),(44,1),(50,5),(42,1),(28,1),(36,1),$
	
	 $(32,2^2),(32,2^4),(40,2^4),(54,3),(27,3^3),(24,2^2),(16,2^2)$, we have
	\[Aut(X_{(n,d)})= \langle g_{(n,d)} \rangle \cong \ZZ/n\ZZ.\] 
   \item For $(n,d)=(28,2^6),(12,1),(16,2^4),(20,2^4)$ we have
   \[\Aut(X)\cong \ZZ/2\ZZ \times \ZZ/n\ZZ.\]
\end{enumerate}
\end{theorem}

\begin{remark}
In \cite[Thm. 1]{oguiso_machida} (1) is proven for $n=66,44,50$ and $40$ by a different method. All other entries of Tables \ref{tbl:real_dets1} and \ref{tbl:real_dets2} have infinite automorphism group. To see this, one checks if the N\'eron Severi lattice is $2$-reflective, which by definition means that the Weyl group is of finite index in its orthogonal group.
One can look up the classification of $2$-reflective lattices in \cite[Prop 0.2.1, Thm 0.2.2]{nikulin:2-reflective}.
\end{remark}
Before proving the Theorems \ref{thm:pair_classification} and \ref{thm:pair_finite_aut} we refine our terminology.\\
A \emph{lattice with isometry} is a pair $(A,a)$, where $A$ is a lattice and $a\in O(A)$ an isometry.
A morphism $\phi: (A,a) \rightarrow (B,b)$ of lattices with isometry is an isometry $\phi: A \rightarrow B$ such that $\phi \circ a =b \circ \phi$.
A primitive extension of lattices with isometry is a morphism
$\phi:(A,a)\oplus (B,b)\hookrightarrow (C,c)$
such that the embedding of $A$ is primitive and $\phi(B) = \phi(A)^\perp$.
Note that since $\phi$ is a morphism, $a \oplus b = c|_{A\oplus B}$.
\begin{definition}
We call two primitive extensions of lattices with isometry
\[(A,a_i)\oplus (B_i,b_i)\hookrightarrow (C_i,c_i), \quad i =1,2\]
isomorphic if there is a commutative diagram
\[
  \begin{tikzcd}[column sep=0.1cm]
     (A_1,a_1) \arrow{d} & \oplus & (B_1,b_1) \arrow{d}\arrow{rrrrrrr}&	& & & & & &(C_1,c_1)\arrow{d}\\
     (A_2,a_2) & \oplus & (B_2,b_2) \arrow{rrrrrrr}&	& & & & & &(C_2,c_2)
  \end{tikzcd}
 \]
 where the vertical arrows are isomorphisms.
\end{definition}

\begin{proposition}
There is a one to one correspondence between isomorphism classes of primitive extensions of
$(A,a) \oplus (B,b)$
and the double coset
	\[
	\Aut(B,b) \setminus \left\{ 
	\begin{matrix}
	\mbox{Glue maps } \phi: G_A \xrightarrow{\sim} G_B\\
	\mbox{satisfying } \phi \circ a = b \circ \phi
	\end{matrix}
	\right\}/\Aut(A,a).
	\]
\end{proposition}
\begin{proof}
 Let 
 \[
  \begin{tikzcd}[column sep=0.1cm]
     (A,a) \arrow{d}{\psi_A} & \oplus & (B,b) \arrow{d}{\psi_B}\arrow{rrrrrrr}&	& & & & & &(C_1,c_1)\arrow{d}{\psi_C}\\
     (A,a) & \oplus & (B,b) \arrow{rrrrrrr}&	& & & & & &(C_2,c_2)
  \end{tikzcd}
 \]
 be an isomorphism of primitive extensions. We can view $C_1$ and $C_2$ as overlattices of $A\oplus B$.
 On the level of glue groups this gives a commutative diagram
  \[
  \begin{tikzcd}[column sep=1cm]
     (G_A,a) \arrow{d}{ \psi_A}  \arrow{r}{\phi_1} & (G_B, b) \arrow{d}{ \psi_B}\\
     (G'_A,a)  \arrow{r}{\phi_2} & (G'_B, b) 
  \end{tikzcd}
 \]
 Then $\phi_2 = \psi_B^{-1} \circ \phi_1 \circ \psi_A$.
 Conversely, if $\phi_2 = g \phi_2 h$ with $g \in \Aut(B,b)$ and $h \in \Aut(A,a)$ is another glue map, then one can check that $g^{-1} \oplus h$ extends to an equivariant isomorphism of the overlattices defined by $\phi_1$ and $\phi_2$.
\end{proof}

\begin{definition}\label{def:few_extensions}
We say that a lattice with isometry $(A,a)$ has \emph{simple glue}\index{Few extensions} if 
\[Aut(A,a) \rightarrow \Aut(q_A,a)=\{g \in O(q_A) \mid g\circ a = a \circ g\}\]
is surjective.
\end{definition}

\begin{example}
Let $(A,a)$ be a lattice  with isometry such that $D_A\cong \FF_p$.
Then $\Aut(D_A,a)=\{\pm id_{D_A}\}$, and we see that $(A,a)$ has simple glue. 
\end{example}

\begin{lemma}
Let $(A,a)$ be a simple $c_n(x)$-lattice. Then $\Aut(A,a)=\langle \pm a \rangle$.
\end{lemma}
\begin{proof}
Let $h \in \Aut(A,a)$. Then $h$ is a $\ZZ[a]$-module homorphism, i.e.
$h \in \ZZ[a]^\times\subseteq K^\times$. Since $h$ is an isometry and $(A,a)$ is simple, we get that 
\[Tr^K_\QQ(hh^\sigma x)=Tr^K_\QQ(x) \quad \forall x \in  K.\]
By non-degeneracy of the trace form, we get $hh^\sigma=1$, i.e. $|h|=1$. By Kronecker's theorem, $h$ is a root of unity.
\end{proof}

\begin{proposition}
Let $(L_0(a),f)$ be a twist of the principal, simple $p(x)$-lattice and $I<\O_K$ such that $D_{L_0(a)} \cong \O_K/I$. Then
\[\Aut(q_{L_0(a)},f) = \{[u] \in (\O_K/I)^\times \mid uu^\sigma \equiv 1 \mod I\}. \]
\end{proposition}
\begin{proof}

Let $g \in \Aut(q_{L_0(a)},f)$. Under the usual identifications $g=[u] \in (\O_K/I)^\times$.
Note that $(f-f^{-1})\O_K=\mathcal{D}^K_k$ is the relative different of $K/k$, and set $d=f-f^{-1}$.
Then $L_0(a)^\vee=\tfrac{1}{ad}\O_K$ and $I=ad\mathcal{O}_K$ (Lemma \ref{lem:discriminant}).
Since $[u]$ preserves the discriminant form, we get that $b(x,y)=b([u]x,[u]y)$ for all $x,y \in \tfrac{1}{ad}\O_K$, i.e.
\[Tr^K_\QQ\left(\frac{1-uu^\sigma}{r'(f+f^{-1})ad^2}\O_K\right) \subseteq \ZZ.\]
By definition of the different, this is equivalent to
\[\frac{1-uu^\sigma}{r'(f+f^{-1})ad^2} \in \mathcal{D}_K^{-1},\]
and consequently
\[1-uu^\sigma \in ad^2r'(f+f^{-1})\mathcal{D}_K^{-1}.\]
Now, the different $\mathcal{D}_K=p'(f)\O_K=(f-f^{-1})r'(f+f^{-1})\O_K$. Hence
\[1-uu^\sigma \in ad\O_K=I\]
as claimed. Conversely, let $u \equiv 1 \mod I$. 
A similar computation shows that the discriminant quadratic form $q_{L_0(a)}$ is preserved 
if and only if \[1-uu^\sigma \in a dd^\sigma \O_k=a (\mathcal{D}^K_k)^2\cap k.\]
However, we already know $1-uu^\sigma \in a\mathcal{D}^K_k \cap k$.
By simplicity, we know that the norm $N(\mathcal{D}^K_k)=|p(1)p(-1)|$ is squarefree, and hence
\[\mathcal{D}^K_k \cap k= (\mathcal{D}_k^K)^2\cap k.\]
\end{proof}
\begin{remark}
 Instead of $|p(-1)p(1)|$ being squarefree, one may assume $K/k$ to be tamely ramified and the proof works.
 However, $\QQ[\zeta_{2^k}]$ is ramified at two and then 
 $\mathcal{D}^K_k \cap k$ is the prime ideal of $\O_k$ above $2$.
 In this case we need the condition $uu^\sigma \equiv 1 \mod I\mathcal{D}^K_k\cap k$.
\end{remark}

\begin{lemma}\label{lem:small_extensions}
	All entries in Table \ref{tbl:possible_dets} except $(24,2^63^4)$ have simple glue. 
\end{lemma}
\begin{proof}
We do the calculation for $(27,3^3)$. The other cases are similar.
Set $\zeta=\zeta_{27}$. Then $I=(1-\zeta)^3$, and $\O_K/I=\ZZ[\zeta]/(1-\zeta)^3$ has $18$ units.
They are given by \[u_\epsilon=\epsilon_0 + \epsilon_1 (1-\zeta)+ \epsilon_2(1-\zeta)^2,\quad (\epsilon_0 \in\{1,2\},\;\epsilon_1,\epsilon_2\in\{0,1,2\})\]
We obtain the equation
\[0 {\equiv}(1-u_\epsilon u_\epsilon^\sigma)\equiv \epsilon_0(\epsilon_1+\epsilon_2)(2-\zeta-\zeta^{-1}) \mod (1-\zeta)^3.\]
We get $6$ distinct solutions for $u_\epsilon$. However $\pm \zeta^k$ for $k \in {1,2,3}$ are all distinct modulo $(1-\zeta)^3$.
The claim follows.
\end{proof}

For $n\in \NN$ we denote by $\mathcal{S}_n$ the symmetric group of $n$ elements and by $\mathcal{D}_n$ the dihedral group - the symmetry group of a regular polygon with $n$ sides.
\begin{lemma}\label{lem:vinberg}
Let $L$ be a hyperbolic lattice. Fix a chamber of the positive cone and denote by
$O^+(L)/W(L)\cong \Gamma(L)\subseteq O^+(L)$ the subgroup generated by the isometries preserving the chamber. Set
\[\phi: \Gamma(L) \rightarrow O(q_L) \quad f \mapsto f|D_L\]
then for $L\neq U(3)\oplus A_2$ in Table \ref{table:sym} $\phi$ is surjective. For  $L=U(3) \oplus A_2$ the cokernel of $\phi$ is generated by $-id$.
It is injective as well for all $L$ in the table except $U(2)\oplus 2D_4$ and $U(2)\oplus D_4 \oplus E_8$ where its kernel is of order $2$. \\
\begin{table}[h]
\caption{Symmetry groups of a chamber}\label{table:sym}
\begin{minipage}{0.3\linewidth}
\centering
{
	\renewcommand{\arraystretch}{1.2}
	\rowcolors{1}{}{lightgray}
	\begin{tabular}{llll}
		\toprule
		$L$				    &$\Gamma(L)$ \\
		\hline
		$U\oplus A_2$   	& $\mathcal{S}_2$\\
		$U(3)\oplus A_2$    & $\mathcal{D}_4$ \\
		$U\oplus 4 A_1$     & $\mathcal{S}_4$ \\
		$U(2)\oplus D_4$    & $\mathcal{S}_5$ \\
		$U\oplus E_8$ 	    & $\mathcal{S}_1$ \\
		$U\oplus D_4 \oplus E_8$ & $\mathcal{S}_3$\\
 	\bottomrule
	\end{tabular}
}
\end{minipage}
\begin{minipage}{0.25\linewidth}
\centering
{
	\renewcommand{\arraystretch}{1.2}
	\rowcolors{1}{}{lightgray}
	\begin{tabular}{ll}
		\toprule
		$L$	&$\Gamma(L)$\\
		\hline
		$U\oplus 2 A_2$     & $\mathcal{D}_4$ \\
		$U(3)\oplus 2 A_2$  & $\mathcal{S}_6 \times \mathcal{S}_2$\\
		$U\oplus D_4$	    & $\mathcal{S}_3$ \\ 
		$U  \oplus E_6$		& $\mathcal{S}_2$ \\
		$U(2)\oplus 2D_4$   & $\mathcal{S}_8 \times \mathcal{S}_2$\\
		$U(2) \oplus D_4 \oplus E_8 $ & $S_5 \times S_2$\\
 	\bottomrule
	\end{tabular}
}
\end{minipage}
\end{table}
\FloatBarrier
\end{lemma}
\begin{proof}
	In all cases we can compute a fundamental root system using Vinberg's algorithm
	\cite[§3]{vinberg:algorithm}.
	An isometry preserves the chamber corresponding to the fundamental root system if and only if it preserves the fundamental root system. We get a sequence
	\[ 0 \rightarrow O^+(L)/W(L)\cong \Gamma(L)\rightarrow Sym(\Gamma)\rightarrow 0\]
	where $Sym(\Gamma)$ denotes the symmetry group of the dual graph of a fundamental root system. Since the fundamental roots form a basis of $L\otimes \QQ$, the sequence is exact. 
	The calculation of $\ker \phi$ is done by computer. For $L=U(2)\oplus 2D_4$ see also \cite[2.6]{kondo:8_points}.
\end{proof}
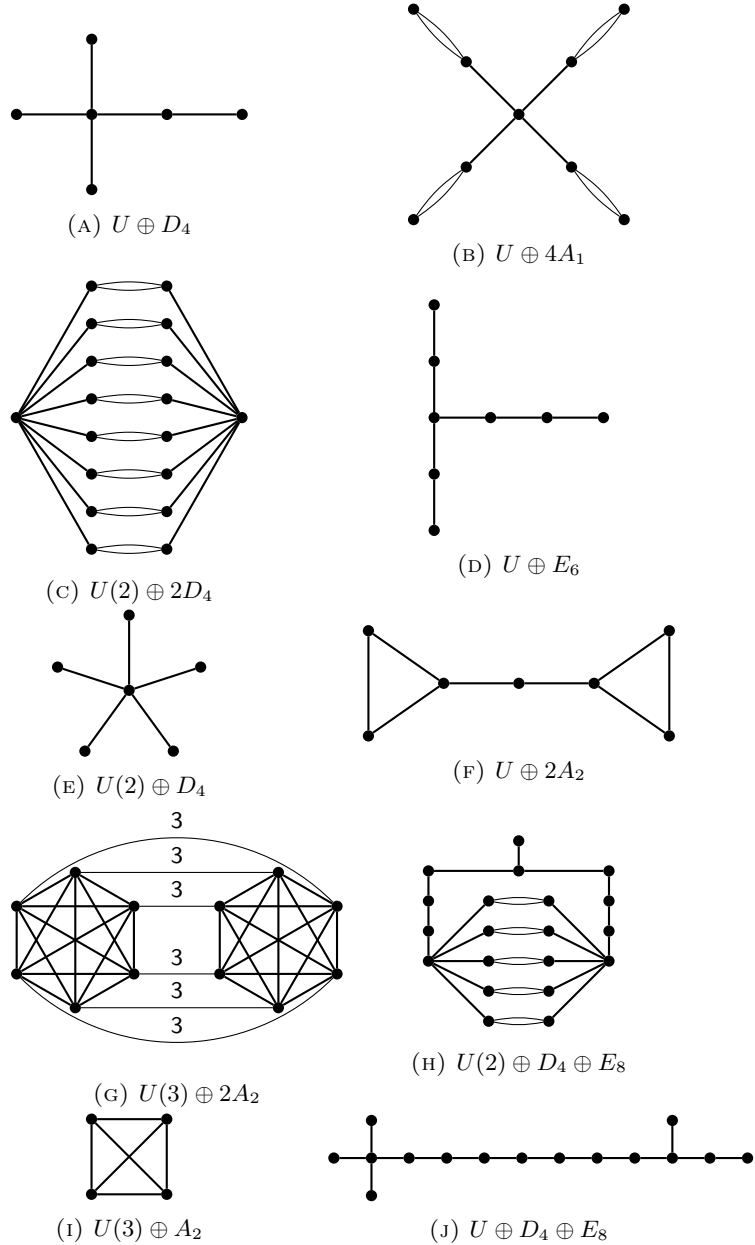
\begin{figure}
\begin{minipage}{0.4\linewidth}
\begin{center}
\begin{tikzpicture}
\tikzset{VertexStyle/.style= {fill=black, inner sep=1.5pt, shape=circle}}	
	\Vertex[NoLabel,x=7,y=1]{5}
	\Vertex[NoLabel,x=6,y=0]{6}
	\Vertex[NoLabel,x=7,y=0]{7}
	\Vertex[NoLabel,x=8,y=0]{8}
	\Vertex[NoLabel,x=9,y=0]{9}
	\Vertex[NoLabel,x=7,y=-1]{10}
		
	\Edges(6,7,8,9)
	\Edges(5,7,10) 
	\end{tikzpicture}
	\subcaption{$U \oplus D_4$}
	\end{center}
\end{minipage}
\begin{minipage}{0.4\linewidth}
	\begin{center}
\begin{tikzpicture}
\tikzset{VertexStyle/.style= {fill=black, inner sep=1.5pt, shape=circle}}	
\Vertex[NoLabel,x=0,y=0]{0}
\Vertex[NoLabel,x=-0.7,y=-0.7]{1}
\Vertex[NoLabel,x=-0.7,y=+0.7]{2}
\Vertex[NoLabel,x=+0.7,y=-0.7]{3}
\Vertex[NoLabel,x=+0.7,y=+0.7]{4}
\Vertex[NoLabel,x=-0.7*2,y=-0.7*2]{5}
\Vertex[NoLabel,x=-0.7*2,y=+0.7*2]{6}
\Vertex[NoLabel,x=+0.7*2,y=-0.7*2]{7}
\Vertex[NoLabel,x=+0.7*2,y=+0.7*2]{8}

\Edges(1,0,3)
\Edges(2,0,4)

\path (1) edge[bend right=10] node [left] {} (5);
\path (2) edge[bend right=10] node [left] {} (6);
\path (3) edge[bend right=10] node [left] {} (7);
\path (4) edge[bend right=10] node [left] {} (8);

\path (1) edge[bend left=10] node [left] {} (5);
\path (2) edge[bend left=10] node [left] {} (6);
\path (3) edge[bend left=10] node [left] {} (7);
\path (4) edge[bend left=10] node [left] {} (8);
\end{tikzpicture}
\subcaption{$U \oplus 4 A_1$}
	\end{center}
\end{minipage}

\begin{minipage}{0.4\linewidth}
	\begin{center}
\begin{tikzpicture}
\tikzset{VertexStyle/.style= {fill=black, inner sep=1.5pt, shape=circle}}	
\Vertex[NoLabel,x=2,y=0]{0a}
\Vertex[NoLabel,x=-1,y=0]{0}
\Vertex[NoLabel,x=0,y=0.5-0.25]{1}
\Vertex[NoLabel,x=0,y=1-0.25]{2}
\Vertex[NoLabel,x=0,y=1.5-0.25]{3}
\Vertex[NoLabel,x=0,y=2-0.25]{4}
\Vertex[NoLabel,x=1,y=0.5-0.25]{5}
\Vertex[NoLabel,x=1,y=1-0.25]{6}
\Vertex[NoLabel,x=1,y=1.5-0.25]{7}
\Vertex[NoLabel,x=1,y=2-0.25]{8}

\Vertex[NoLabel,x=0,y=-0.5+0.25]{1a}
\Vertex[NoLabel,x=0,y=-1+0.25]{2a}
\Vertex[NoLabel,x=0,y=-1.5+0.25]{3a}
\Vertex[NoLabel,x=0,y=-2+0.25]{4a}
\Vertex[NoLabel,x=1,y=-0.5+0.25]{5a}
\Vertex[NoLabel,x=1,y=-1+0.25]{6a}
\Vertex[NoLabel,x=1,y=-1.5+0.25]{7a}
\Vertex[NoLabel,x=1,y=-2+0.25]{8a}

\Edges(1,0,3)
\Edges(2,0,4)
\Edges(1a,0,3a)
\Edges(2a,0,4a)

\Edges(5,0a,7)
\Edges(6,0a,8)
\Edges(5a,0a,7a)
\Edges(6a,0a,8a)

\path (1) edge[bend right=10] node [left] {} (5);
\path (2) edge[bend right=10] node [left] {} (6);
\path (3) edge[bend right=10] node [left] {} (7);
\path (4) edge[bend right=10] node [left] {} (8);

\path (1) edge[bend left=10] node [left] {} (5);
\path (2) edge[bend left=10] node [left] {} (6);
\path (3) edge[bend left=10] node [left] {} (7);
\path (4) edge[bend left=10] node [left] {} (8);

\path (1a) edge[bend right=10] node [left] {} (5a);
\path (2a) edge[bend right=10] node [left] {} (6a);
\path (3a) edge[bend right=10] node [left] {} (7a);
\path (4a) edge[bend right=10] node [left] {} (8a);

\path (1a) edge[bend left=10] node [left] {} (5a);
\path (2a) edge[bend left=10] node [left] {} (6a);
\path (3a) edge[bend left=10] node [left] {} (7a);
\path (4a) edge[bend left=10] node [left] {} (8a);

\end{tikzpicture}
\subcaption{$U(2) \oplus 2D_4$}
	\end{center}
\end{minipage}
\begin{minipage}{0.4\linewidth}
	\begin{center}
	\begin{tikzpicture}
	\tikzset{VertexStyle/.style= {fill=black, inner sep=1.5pt, shape=circle}}
	\Vertex[NoLabel,x=0,y=-2*0.75]{1}
	\Vertex[NoLabel,x=0,y=-1*0.75]{2}
	\Vertex[NoLabel,x=0,y=0]{3}
	\Vertex[NoLabel,x=0,y=1*0.75]{4}
	\Vertex[NoLabel,x=0,y=2*0.75]{5}
	\Vertex[NoLabel,x=1*0.75,y=0*0.75]{6}
	\Vertex[NoLabel,x=2*0.75,y=0]{7}
	\Vertex[NoLabel,x=3*0.75,y=0]{8}
	\Edges(1,2,3,4,5)
	\Edges(3,6,7,8) 
	\end{tikzpicture}
	\subcaption{$U \oplus E_6$}
	\end{center}
\end{minipage}
\begin{minipage}{0.4\linewidth}
\begin{center}
	\begin{tikzpicture}
	\tikzset{VertexStyle/.style= {fill=black, inner sep=1.5pt, shape=circle}}
	\Vertex[NoLabel,x=0,y=1]{1}
	\Vertex[NoLabel,x=-0.95105,y=0.30901]{2}
	\Vertex[NoLabel,x=-0.587785,y=-0.809016]{3}
	\Vertex[NoLabel,x=0.587785,y=-0.809016]{4}
	\Vertex[NoLabel,x=0.95105,y=0.30901]{5}
	\Vertex[NoLabel,x=0,y=0]{0}
	\Edges(1,0,2)
	\Edges(3,0,4)
	\Edges(0,5) 
	\end{tikzpicture}
	\subcaption{$U(2) \oplus D_4$}
\end{center}
\end{minipage}
\begin{minipage}{0.4\linewidth}
\begin{center}
	\begin{tikzpicture}
	\tikzset{VertexStyle/.style= {fill=black, inner sep=1.5pt, shape=circle}}
	\Vertex[NoLabel,x=-2,y=0.7]{1}
	\Vertex[NoLabel,x=-2,y=-0.7]{2}
	\Vertex[NoLabel,x=-1,y=0]{5}
	\Vertex[NoLabel,x=0,y=0]{6}
	\Vertex[NoLabel,x=1,y=0]{7}
	\Vertex[NoLabel,x=2,y=0.7]{3}
	\Vertex[NoLabel,x=2,y=-0.7]{4}
	\Edges(1,2,5,1)
	\Edges(3,4,7,3)
	\Edges(5,6,7) 
	\end{tikzpicture}
	\subcaption{$U \oplus 2A_2$}
	\end{center}
\end{minipage}
\begin{minipage}{0.4\linewidth}
\begin{center}
	\begin{tikzpicture}[scale=0.9]
	\tikzset{VertexStyle/.style= {fill=black, inner sep=1.5pt, shape=circle}}
	\Vertex[NoLabel,y=0.5,x=0.866]{1}
	\Vertex[NoLabel,y=-0.5,x=0.866]{2}
	\Vertex[NoLabel,y=-1,x=0]{3}
	\Vertex[NoLabel,y=-0.5,x=-0.866]{4}
	\Vertex[NoLabel,y=0.5,x=-0.866]{5}
	\Vertex[NoLabel,y=1,x=0]{6}
	\Edges(1,2)
	\Edges(1,3)
	\Edges(1,4)
	\Edges(1,5)
	\Edges(1,6)
	\Edges(2,3)
    \Edges(2,4)
	\Edges(2,5)
	\Edges(2,6)
    \Edges(3,4)
	\Edges(3,5)
	\Edges(3,6)
	\Edges(4,5)
	\Edges(4,6)
	\Edges(5,6)

	\Vertex[NoLabel,y=0.5,x=3+0.866]{1a}
	\Vertex[NoLabel,y=-0.5,x=3+0.866]{2a}
	\Vertex[NoLabel,y=-1,x=3+0]{3a}
	\Vertex[NoLabel,y=-0.5,x=3-0.866]{4a}
	\Vertex[NoLabel,y=0.5,x=3-0.866]{5a}
	\Vertex[NoLabel,y=1,x=0+3]{6a}
	\Edges(1a,2a)
	\Edges(1a,3a)
	\Edges(1a,4a)
	\Edges(1a,5a)
	\Edges(1a,6a)
	\Edges(2a,3a)
    \Edges(2a,4a)
	\Edges(2a,5a)
	\Edges(2a,6a)
    \Edges(3a,4a)
	\Edges(3a,5a)
	\Edges(3a,6a)
	\Edges(4a,5a)
	\Edges(4a,6a)
	\Edges(5a,6a)
	
	\path[every node/.style={font=\sffamily\small}]
        (6) edge [above] node {3} (6a)
        (3) edge [above] node {3} (3a)
        (1) edge [above] node {3} (5a)
        (2) edge [above] node {3} (4a)
        (5) edge [above, bend left=45] node {3} (1a)
        (4) edge [above, bend right=45] node {3} (2a);
	\end{tikzpicture}
	\subcaption{$U(3) \oplus 2A_2$}
	\end{center}
\end{minipage}
\begin{minipage}{0.3\linewidth}
	\begin{center}
		\begin{tikzpicture}[scale=0.8]
		\tikzset{VertexStyle/.style= {fill=black, inner sep=1.5pt, shape=circle}}	
		\Vertex[NoLabel,x=-1,y=-1]{1}
		\Vertex[NoLabel,x=0 ,y=-1]{1a}
		\Vertex[NoLabel,x=-1,y=-0.5]{2}
		\Vertex[NoLabel,x=0 ,y=-0.5]{2a}
		\Vertex[NoLabel,x=-1,y= 0]{3}
		\Vertex[NoLabel,x=0 ,y= 0]{3a}
		\Vertex[NoLabel,x=-1,y= 0.5]{4}
		\Vertex[NoLabel,x=0 ,y= 0.5]{4a}
		\Vertex[NoLabel,x=-1,y= 1]{5}
		\Vertex[NoLabel,x=0 ,y= 1]{5a}

		\Vertex[NoLabel,x=-2 ,y=0]{0}
		\Vertex[NoLabel,x=1 ,y= 0]{0a}
		
		\Vertex[NoLabel,x=-2 ,y=0.5]{6}
		\Vertex[NoLabel,x=1 ,y=0.5]{6a}
		\Vertex[NoLabel,x=-2 ,y=1]{7}
		\Vertex[NoLabel,x=1 ,y= 1]{7a}
		\Vertex[NoLabel,x=-2 ,y=1.5]{8}
		\Vertex[NoLabel,x=1 ,y= 1.5]{8a}
		\Vertex[NoLabel,x=-0.5 ,y=1.5]{9}
		\Vertex[NoLabel,x=-0.5 ,y= 2]{10}

		\path (1) edge[bend right=10] node [left] {} (1a);
		\path (2) edge[bend right=10] node [left] {} (2a);
		\path (3) edge[bend right=10] node [left] {} (3a);
		\path (4) edge[bend right=10] node [left] {} (4a);
		\path (5) edge[bend right=10] node [left] {} (5a);
		
		\path (1) edge[bend left=10] node [left] {} (1a);
		\path (2) edge[bend left=10] node [left] {} (2a);
		\path (3) edge[bend left=10] node [left] {} (3a);
		\path (4) edge[bend left=10] node [left] {} (4a);
		\path (5) edge[bend left=10] node [left] {} (5a);
		
		\Edges(1,0,2)
		\Edges(3,0,4)
		\Edges(5,0)
		\Edges(1a,0a,2a)
		\Edges(3a,0a,4a)
		\Edges(5a,0a)
		\Edges(0,6,7,8,9,8a,7a,6a,0a)
		\Edges(9,10)
				
		\end{tikzpicture}
		\subcaption{$U(2) \oplus D_4 \oplus E_8$}
	\end{center}
\end{minipage}
\begin{minipage}{0.4\linewidth}
	\begin{center}
		\begin{tikzpicture}
		\tikzset{VertexStyle/.style= {fill=black, inner sep=1.5pt, shape=circle}}
		\Vertex[NoLabel,x=0,y=0]{1}
		\Vertex[NoLabel,x=1,y=0]{2}
		\Vertex[NoLabel,x=0,y=1]{3}
		\Vertex[NoLabel,x=1,y=1]{4}
		\Edges(1,2,3,4,1)
		\Edges(1,3)
		\Edges(2,4)
		\end{tikzpicture}
		\subcaption{$U(3) \oplus A_2$}
	\end{center}
\end{minipage}
\begin{minipage}{0.4\linewidth}
	\begin{center}
		\begin{tikzpicture}
		\tikzset{VertexStyle/.style= {fill=black, inner sep=1.5pt, shape=circle}}
		\Vertex[NoLabel,x=0,y=0]{1}
		\Vertex[NoLabel,x=0.5,y=0]{2}
		\Vertex[NoLabel,x=1,y=0]{3}
		\Vertex[NoLabel,x=1.5,y=0]{4}
		\Vertex[NoLabel,x=2,y=0]{5}
		\Vertex[NoLabel,x=2.5,y=0]{6}
		\Vertex[NoLabel,x=3,y=0]{7}
		\Vertex[NoLabel,x=3.5,y=0]{8}
		\Vertex[NoLabel,x=4,y=0]{9}
		\Vertex[NoLabel,x=4.5,y=0]{10}
		\Vertex[NoLabel,x=5,y=0]{11}
		\Vertex[NoLabel,x=5.5,y=0]{12}
		\Vertex[NoLabel,x=4.5,y=0.5]{13}
		\Vertex[NoLabel,x=0.5,y=0.5]{14}
		\Vertex[NoLabel,x=0.5,y=-0.5]{15}
		\Edges(1,2,3,4,5,6,7,8,9,10,11,12)
		\Edges(14,2,15)
		\Edges(13,10)
		\end{tikzpicture}
		\subcaption{$U \oplus D_4 \oplus E_8$}
	\end{center}
\end{minipage}
\caption{Dynkin diagrams of the fundamental root systems}
\end{figure}
\begin{proof}[Proof of Theorem \ref{thm:pair_classification}]
	Fix some pair $(X,G)$ as in the theorem and write $G=\langle g \rangle$ for
	$g \in G$ such that $g^*\omega=\zeta_n \omega$. In order to prove the theorem,
	we have to show that the pair $(H^2(X,\ZZ),g)$ is unique up to isomorphism.
	We have seen that $n$ and $d=\det T$ determine $(T,g|_T)$ (and thus $X$ by Thm. \ref{thm:main}) up to isomorphism.
	By Lemma \ref{lem:small_extensions}, $(T,g|_T)$ has simple glue. Hence
	the isomorphism class of $(H^2(X,\ZZ),g)$ is determined by the isomorphism class of 
	$(\NS,g|_{\NS})$. What remains is to determine all possible isomorphism classes of 
	$(\NS,g|_{\NS})$ for $(n,d)$ fixed. This is done in the following lemmas.
\end{proof}
	
\begin{lemma}
For \boldmath$(p,p)$\unboldmath, $p=13,17,19$, $g|_{\NS}$ is the identity.

\end{lemma}
\begin{proof}
Since the order of $g$ on $\NS$ is strictly smaller than $n=p$ in these cases, it can only be one. 
\end{proof}

\begin{proof}[Proof of Theorem \ref{thm:pair_finite_aut}]
(1) By Lemma \ref{lem:vinberg}, we have for these lattices that 
\[O^+(\NS)/W(\NS)\cong O(q_{\NS})\cong O(q_{\Trans}).\] 
Consequently, automorphisms are determined by their action on the transcendental lattice and this group is generated by $g_{(n,d)}$. 
(2) In this case $\phi: \Gamma(\NS) \rightarrow O(q_{\NS})$ has a kernel of order two and there are exactly two possibilities for $g|\NS$. They differ by an element of the kernel corresponding to a symplectic automorphism of order two.
\end{proof}

We note the following theorem on the rational square class of a lattice with isometry for later use.
\begin{theorem}\cite[3.3.14 (ii)]{nebe:habil}\label{thm:det_is_square}
Let $(L,f)$ be a $c_n(x)^m$-lattice and $f$ of order $n$.
\[\det L \in 
\begin{cases}
p^m \cdot \left(\QQ^\times\right)^2, &\mbox{ for } n=p^k, p\neq 2,\\
\left(\QQ^\times\right)^2, &\mbox {  else.}
\end{cases}\]
\end{theorem}

Let $X$ be a K3 surface and $g\in \Aut(X)$ of finite order $n$.
and note that since $g|\NS$ is of finite order $g|\NS$ preserves a chamber of the positive cone if and only if
\[\ker \left(g^{n-1}+g^{n-2} \cdots +1\right)|_{\NS}\]
is root free (\cite{mcmullen:minimum} calls such roots cyclic).
We turn this into a definition.
\begin{definition}
 Let $N$ be a hyperbolic or positive definite lattice and $g \in O(N)$ an isometry of finite order $n$.
 We call $g$ \emph{unobstructed} if
 \[\ker \left(g^{n-1}+g^{n-2} \cdots +1\right)\]
is root free.  Else we call $g$ \emph{obstructed}.
\end{definition}
In the following we set
\[C_i=\ker c_i(g|H^2(X,\ZZ)),\quad C_iC_j=\ker c_ic_j(g|H^2(X,\ZZ)).\]
Recall that for $n \in \NN$ we say that two lattices glue along $n$ if
$n$ is the order of the glue.
\begin{lemma}
For \boldmath$(38,19)$ \unboldmath 
there are two pairs $(X,g_1)$, $(X,g_2)$ up to isomorphism.
Set 
\[R_1=
\left(
\begin{matrix}
-2&-1\\
-1&-10
\end{matrix}
\right), \quad 
R_2 = 
\left(
\begin{matrix}
 -8& 2\\
 2& -10
\end{matrix}
\right).
\]
Then $\NS\cong U \oplus R_1$ and
$g_1|_{\NS}$ is given by the gluing of 
\[C_1 \cong U \oplus (-2) \mbox{ and } C_2 \cong (-38)\]
along $2$.
Further $g_2|{\NS}$ is given by the gluing of 
\[C_1 \cong (2) \oplus (-2) \mbox{ and } C_2 \cong R_2\]
along $2^2$.
\end{lemma}
\begin{proof}
There are $3$ cases 
\[\chi_{g|\NS}=(x-1)^r(x+1)^{4-r}, \quad (r=1,2,3).\]
Note that $C_{1}$ is $2$-elementary and $\det C_{1} \mid 2^m$ where $m=min\{r,4-r\}$ (Thm \ref{thm:glue_structure}).
\begin{enumerate}
\item[$r=1$:] Here $C_1=(2)$ and $C_2=(-2)\oplus R_1$ which is the unique even, negative definite lattice of determinant $-38$.
\item[$r=2$:] We see $\det C_2 \mid 2^219$ and there are 4 such lattices -
$R_1$,$(-2)\oplus (-38)$, $R_1(2)$ and $R_2$. The first two lattices have roots and $R_1(2)$ has wrong $19$ glue, since the Legendre symbol $\left(\frac{2}{19}\right)=-1$. We are left with $R_2$. The glue is easily seen to be unique.
\item[$r=3$:] We have the single choice $C_2=(-38)$ and $C_1=U \oplus (-2)$. Indeed the gluing exists and is unique. 
\end{enumerate}
\end{proof}

\begin{lemma}
\boldmath $(34,17)$ \unboldmath 
There are two pairs $(X,g_1)$ and $(X,g_2)$ for $(34,17)$.
The action of $g_1|_{\NS}$ is given by the gluing of
\[C_{1}=U \oplus (-2) \oplus (-2) \quad \mbox{ and } \quad C_2=-\left(
\begin{matrix}
6 & 2\\
2 & 12
\end{matrix}
\right).
\]
along $2^2$.

The action of $g_2|_{\NS}$ is given by the gluing of 
\[C_{1}=(2)\oplus (-2) \oplus (-2) \quad \mbox{ and }\quad C_2= 
-2\left(
\begin{matrix}
2 & 1 & 1 \\
1 & 3 & 1 \\
1 & 1 & 4
\end{matrix}\right)\]
along $2^3$.

\end{lemma}

\begin{proof}
Here we need a little more work.
Note that $\NS$ is in the genus  $\upperRomannumeral{2}_{(1,5)}(17^{-1})$.
There are the $5$ cases
\[\chi_{g|\NS}=(x-1)^r(x+1)^{6-r}, \qquad r \in \{1,\dots 5\}.\]
In any case $17 \mid \det C_2 \mid 2^m 17$ where $m=\min\{r,6-r\}$.
The $17$ part of the genus symbol of $C_2$ is $17^{-1}$, and moreover $2 (D_{C_2})_2=0$. Then the genus symbol of $C_2$ is $\upperRomannumeral{2}_{(0,6-r)}(2^v_*17^{-1})$ where the asterisk is an unknown entry.
\begin{enumerate}
\setlength{\itemsep}{10pt}
\item[$r=1$:] Then $C_{1}=(2)$ and $\det C_2=-34$. Hence in order to glue above $2$, $C_2$ must belong to $\upperRomannumeral{2}_{(0,5)}(2_7 17^{-1})$ or
$\upperRomannumeral{2}_{(0,5)}(2_3^{-1} 17^{-1})$, but both genera are empty as they contradict the oddity formula \cite[Chapter 15, (16) ]{conway_sloane:sphere_packings}.
\item[$r=2$:] Here $C_2$ is even of signature $(0,4)$ and determinant $d=-17,-2\cdot 17, -4 \cdot 17$.
Looking at the tables in \cite{nipp:table_quaternary_forms}, we see that there
are $1,0,7$ such forms, and all of them contain roots.
\item[$r=3$:]
From the tables in \cite{brandt-intrau:table_ternary_forms}, we extract the following.
If $v=0,\pm 2$ the respective genera are empty.
If $v=1$, there is a single genus, namely $\upperRomannumeral{2}_{(0,3)}(2_1^1 17^{-1})$
containing two classes - both have maximum $-2$.
However, for $v=3$ there are $9$ negative definite ternary forms of determiant $2^3 17$. Only a single one of them 
has the right 2-genus symbol and no roots. It is given by
\[C_2=
-2\left(
\begin{matrix}
2 & 1 & 1 \\
1 & 3 & 1 \\
1 & 1 & 4
\end{matrix}\right).
\]
Indeed, here $C_{1}=(2) \oplus (-2) \oplus (-2)$ works just fine,
and as $|O(q_{C_{1}})|=2$ it is evident that the gluing is unique as well. 
\item[$r=4$:] Here $\det C_2 \mid 2^2 17$, and  

we get the possibilities 
\[
-\left(
\begin{matrix}
2 & 0\\
0 & 34
\end{matrix}
\right),
-\left(
\begin{matrix}
4 & 2\\
2 & 18
\end{matrix}
\right),
-\left(
\begin{matrix}
6 & 2\\
2 & 12
\end{matrix}
\right).
\]
The first two have wrong $17$ glue.
We are left with the third one. 
It has \[(q_R)_2\cong (1/2) \oplus (1/2).\] 
Then there is the single possibility $C_{1}\cong U\oplus (-2)\oplus (-2)$. Surjectivity of 
$O(C_{1})\rightarrow O(q_{C_{1}})$, hence uniqueness of the extension is provided by Theorem \ref{thm:orthogonal_group_surjective}.
\item[$r=5$:]
Here $C_2=(-34)\in \upperRomannumeral{2}_{(0,1)}(2_7 17)$ has wrong $17$-glue.
\end{enumerate}
\end{proof}

\textbf{The holomorphic Lefschetz formula.} 
For the next lemma we use the holomorphic (see \cite[p.542]{atiyah:index2} and \cite[p.567]{atiyah:index3}) and topological Lefschetz formula.
We give a short account. See \cite{taki:3power} for a similar application.\\

Recall that $g$ is a purely non-symplectic automorphism of the K3 surface $X$ with $g^*\omega=\zeta_n\omega$, where $0 \neq \omega \in H^0(X,\Omega_X^2)$. Let $x$ be a fixed point of $g$. 
Then the local action of $g$ at $x$ can be linearized and diagonalized (in the holomorphic category). 
\begin{definition}
A fixed point is said to be of type $(i,j)$ if the action is locally of the form
\[
\left(
\begin{matrix}
\zeta_n^i & 0\\
0 & \zeta_n^j
\end{matrix}
\right).
\]
\end{definition}

This implies that the fixed point set $X^g$ is the disjoint union of isolated fixed points and smooth curves $C_1, \dots, C_N$.
Set 
\[a_{ij}=\frac{1}{(1-\zeta_n^i)(1-\zeta_n^j)} \quad \mbox{ and } \quad b(g)=\frac{(1+\zeta_n)(1-g)}{(1-\zeta_n)^2}.\]
Denote by $m_{i,j}$ the number of isolated fixed points of type $(i,j)$,
and set $g_l=g(C_l)$ the genus of the fixed curve $C_l$.
The topological Lefschetz formula is
\[e(X^g)=\sum_{i=0}^4 (-1)^iTr(g^*|H^i(X,\ZZ))\]
which in our setting amounts to
\[M+\sum_{l=1}^N \left( 2-2g(C_l)\right)=2 + Tr (g^*|T)+ Tr (g^*|\NS)\]
where $M=\sum_{\substack{i+j=n+1\\1<i\leq j< n}} m_{i,j}$ is the number of isolated fixed points.
The holomorphic Lefschetz formula is
\[1+\overline{\zeta_n}=\sum_{i=0}^2 (-1)^i Tr(g^*|H^i(X,\O_X))=\sum_{\substack{i+j=n+1\\1<i\leq j< n}} a_{ij}m_{ij} + \sum_{l=1}^N b(g_l).\]
\begin{lemma}
For \boldmath$(26,13)$ \unboldmath there are three pairs $(X,f_1)$, $(X,f_2)$ and $(X,f_3)$.
The action of $g_1|_{\NS}$ is given by
the gluing of
\[C_{1}=U \oplus D_4 \oplus A_1 \quad \mbox{and} \quad
C_2=
-\left(
\begin{matrix}
4&  2&  2\\
2&  4&  2\\
2&  2& 10
\end{matrix}
\right)\]
along $2^3$.
The action of $g_2|_{\NS}$ is given by
the gluing of
\[C_{1} \cong (2) \oplus E_8 \mbox{ and } C_2 \cong (-26)\]
along $2$.

\textcolor{red}{
The action of $g_3|_{\NS}$ is given by the glueing of
\[C_1 = U(2) \oplus D_4 \text{ and } C_2 = -\begin{pmatrix}
4 & 2 & 2 & 2 \\
2 & 4 & 2 & 2 \\
2 & 2 & 4 & 2 \\
2 & 2 & 2 & 8
\end{pmatrix}\]
along $2^4$.}
\end{lemma}
\begin{proof}
We already know the uniqueness of $(X,g^2)$. One can check that $g^2$ has $9$ isolated fixed points and a (pointwise) fixed curve of genus $0$.
By \cite[8.4]{artebani_sarti_taki:non-symplectic} their local types are given by
\[m_{2,12}=3,m_{3,11}=3,m_{4,10}=2,m_{5,9}=1.\]
Since $X^g\subseteq X^{g^2}$, either $g$ fixes a curve of genus $0$ and at most $9$ isolated points, or $g$ does not fix a curve and at most $11$ points.

A calculation of the holomorphic Lefschetz formula yields the following possibilities:
$X^g$ fixes a curve of genus zero and $7$ or $9$ points, or
$X^g$ fixes $4,5,6$ or $7$ points and no curve.
In any case the fixed locus has Euler characteristic $4 \leq e(X^g) \leq 11$.
Write
\[\chi_{g|\NS}=(x-1)^r(x+1)^{10-r}\]
for the characteristic polynomial of the action of $g$ on $\NS$.
Then the topological Lefschetz formula reads
\[e(X^g)=2+Tr(g^*|T)+ Tr(g^*|\NS)=2+1+r-(10-r)=2r-7,\]
and consequently $6\leq r \leq 9$
We view $\NS$ as a primitive extension of $C_{1}\oplus C_2$. 
Since $\res(c_1,c_{26})=1$, we see that $13 \mid \det C_2$. Further, $C_{1}$ is $2$-elementary.
We conclude that $|\!\det C_2|=2^k 13$ where $k\leq \min \{r,10-r\}$.

\begin{enumerate}
\item[$r=6$:]
Looking at the tables in \cite{nipp:table_quaternary_forms}, we see that
for $k=0,1,2,3$ all even forms of signature $(0,4)$ and determinant $2^k13$ have roots.
\textcolor{red}{For $k=4$ there are three genera of determinant $2^4 \cdot 13$. They give a total of $5$ even forms.
Exactly one of them has no roots. It results in $g_3$.}
\item[$r=7$:]
Here we use the tables of \cite{brandt-intrau:table_ternary_forms} to list even forms of 
signature $(0,3)$ and determinant $2^k13$.
\begin{itemize}	
\item For $k=0$ there is no lattice of this determinant.
\item For $k=1$ there is a single class, but it is obstructed.
\item For $k=2$ there are two genera of this determinant, but their $2$ discriminant group is isomorphic to $\ZZ/4\ZZ$.
\item For $k=3$ there is a single genus with right $2$ discriminant and $13$ glue.
It is 
\[\upperRomannumeral{2}_{(0,3)}(2^{-3}_1 13^{-1})\]
and consists of the two classes 
\[
-\left(
\begin{matrix}
2&  0&  0\\
0&  2&  0\\
0&  0& 26
\end{matrix}
\right), \quad \mbox{ and } \quad
-\left(
\begin{matrix}
4&  2&  2\\
2&  4&  2\\
2&  2& 10
\end{matrix}
\right) \cong C_2
\]
one of which, $C_2$, has no roots. 
Then \[C_1\cong U \oplus D_4 \oplus A_1.\]
\end{itemize}
\item[$r=8$:] There are no negative definite lattices of rank $2$ and determinant $13$ or $26$. 
But there are two of determinant $2^2 \cdot 13$:
\[
\left(
\begin{matrix}
-2&  0\\
0&  -26
\end{matrix}
\right),
\quad \mbox{ and } \quad
\left(
\begin{matrix}
-4&  2 \\
2&  -14  
\end{matrix}
\right).
\]
The first one is obstructed while the second one has wrong $19$-glue.
\item[$r=9$:] Here we can take $C_2 \cong (-26)$ and $C_{1} \cong (2) \oplus E_8$. The lattices glue.

We have to check uniqueness of the gluings. This is provided by the surjectivity of
\[O(C_{1})\rightarrow O(q_{C_{1}})\]
which follows in each case from \cite[1.14.2]{nikulin:quadratic_forms}.\\
\end{enumerate}
\end{proof}

\begin{lemma}\label{lem:3^6 2^3}
For \boldmath$(36,2^63^2)$\unboldmath, the characteristic polynomial is
\[\chi_g=c_{36}c_{18}c_{4}c_2c_1\]
and the gluings are given by the following diagram. Edges correspond to glue maps between the respective sublattices and they are decorated with the order of the glue.
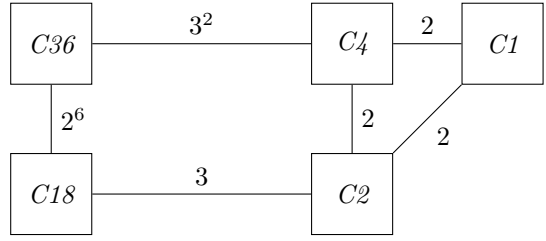
\begin{center}  
	%\begin{figure}
	\tikzstyle{block} = [draw, rectangle, minimum height=3em, minimum width=3em]
	\tikzstyle{virtual} = [coordinate]
	
	\begin{tikzpicture}[auto, node distance=2cm]
	% Place nodes
	\node [block]                 (C36)     {$C_{36}$};
	\node [virtual, right of=C36](up){};
	\node [block, right of=up]   (C4)     {$C_4$};
	\node [block, below of=C36] (C18)    {$C_{18}$};
	\node [block, below of=C4] (C2)  {$C_2$};
	\node [block, right of=C4] (C1)  {$C_{1}$};
	% Connect nodes
	\draw [-] (C36) -- node {$2^6$} (C18);
	\draw [-] (C36) -- node {$3^2$}(C4);
	\draw [-] (C4)  -- node {$2$}(C2);
	\draw [-] (C4)  -- node {$2$}(C1);
	\draw [-] (C1)  -- node {$2$}(C2);
	\draw [-] (C18)  -- node {$3$}(C2);
	\end{tikzpicture}
\end{center}
This determines $(\NS,g|\NS)$ uniquely up to isomorphism.
\end{lemma}
\begin{proof}
The possible contributors to the resultant are $c_9,c_{18},c_4$ and $c_{12}$.
First the $2^6$ contribution is coming from either $c_9$ or $c_{18}$ dividing $\chi(g|\NS)$.
Then there is no room for $c_{12}$ left. Thus the $3^2$ contribution is coming from $c_4$.
This leaves us with 
\[\chi_g=c_{36}c_{18}c_{4}(x\pm 1) (x-1) \quad \mbox{or}\quad c_{36}c_{9}c_{4}(x\pm 1) (x-1).\]
Since the principal $c_4(x)$-lattice has determinant $2^2$, we have to glue it over $2^2$. This determines the characteristic polynomial to be $c_{36}c_{18}c_4c_2c_1$ or $c_{36}c_{9}c_4c_2c_1$.
At this point we know $C_{36}$ ,$C4\cong (-6)\oplus(-6)$, $C_{18}\cong E_6(2)$ and their gluings which exist by \cite[Theorem 3.1]{mcmullen:entropy_and_glue}.
Then 
\[\left(q_{C_{1}C_2}\right)_3 \cong (q_{E6(2)})_3(-1)\cong(2/3).\]
The case $C_{1}C_2=C_{1}\oplus C_2$ leads to $C_2=(-2)$ which is obstructed or $C_2=(-6)$ which has the wrong $3$-glue. Thus we have to glue. 
Then $C_{1}\cong (4)$ and $C_2 \cong (-12)$ as $C_{1}\cong (12)$ has wrong $3$-glue.
This gluing is unique since $(\ZZ/4\ZZ)^\times=\{\pm 1 \}$.
Since $(D_{C_2})_3$ can be glued to $C_{18}$ but not to $C_{9}$, we have 
\[\chi_g=c_{36}c_{18}c_4c_2c_1.\]
The only step at which we have non-trivial freedom in the choice of glue is between $C_{1}C_2$ and $C_4$. This freedom is due to the action of $g|D_{C_4}$. Thus is does not affect the isomorphism class of $(C_{1}C_2C_4,g_1 \oplus g_2\oplus g_4)$ and uniqueness of $(\NS,g)$ up to isomorphism follows. 
\end{proof}

The proof of the following Lemmas is similar to what we have already seen. We refer the reader to the authors thesis \cite{brandhorst:thesis} for the computational details.
\begin{lemma}
For \boldmath$(36,3^4)$ \unboldmath the action of $g|_{\NS}$ is uniquely determined by the following gluing diagram.
\begin{center}  
	%\begin{figure}
	\tikzstyle{block} = [draw, rectangle, minimum height=3em, minimum width=3em]
	\tikzstyle{virtual} = [coordinate]
	
	\begin{tikzpicture}[auto, node distance=2cm]
	% Place nodes
	\node [block]                 (C36)     {$C_{36}$};
	\node [block, right of=C36]   (C12)     {$C_{12}$};
	\node [block, right of=C12] (C3)    {$C_3$};
	\node [block, right of=C3] (C1)  {$C1\cong U\oplus A_2$};
	% Connect nodes
	\draw [-] (C36) -- node {$3^4$} (C12);
	\draw [-] (C12) -- node {$2^2$}(C3);
	\draw [-] (C3)  -- node {$3$}(C1);
	\end{tikzpicture},
\end{center}
\end{lemma}

\begin{lemma}\label{lem:21_7^2_classification} \boldmath$(21,7^2),(42,7^2)$ \unboldmath
On the K3 surface $X_{(21,7^2)}$ there are $3$ (resp. $2$) conjugacy classes of purely non-symplectic automorphisms of order $21$ (resp. $42$). They are distinguished by their invariant lattices 
 \begin{enumerate}
 	\item $C_1\cong U\oplus E_6$, 
 	\item $C_1\cong U \oplus 2 A_2$,
 	\item $C_1\cong U(3) \oplus A_2$.
 \end{enumerate}
Only $(1)$ and $(2)$ are the square of an automorphism of order $42$.
They 
\end{lemma}

\begin{lemma}\label{lem:21,7^2}
Affine Weierstraß models for $X_{(21,7^2)}$ and the automorphisms of order $21$ and $42$ corresponding to the cases (1),(2) in Lemma \ref{lem:21_7^2_classification}
are given below. For case (3) there is a singular projective model.\\
\begin{enumerate}
\item    $y^2=x^3+t^4(t^7+1), \qquad \quad \quad (x,y,t)\mapsto(\zeta_3 \zeta_7^6 x,\pm \zeta_7^2 y,\zeta_7 t)$; 
\item    $y^2=x^3+t^3(t^7+1), \qquad \quad \quad (x,y,t)\mapsto (\zeta_3 \zeta_7^3 x,\pm \zeta_7 y,\zeta_7t)$;
\item    $x_0^3x_1 + x_1^3x_2 + x_0x_2^3 - x_0x_3^3, \quad \!(x_0,x_1,x_2,x_3) \mapsto(\zeta_7x_0,\zeta_7x_1,x_2,\zeta_3 x_3).$
\end{enumerate}
\end{lemma}
\begin{proof}
One can identify the three cases by computing the fixed lattice of $f^{14}$. 
\end{proof}

\section{Acknowledgements}
I thank my advisor Matthias Schütt for his guidance, innumerable helpful comments and discussions,
Davide Veniani for sharing his insights on finding complex models, 
Víctor Gonzalez-Alonso for fruitful discussions on gluings and lattices
and Daniel Loughran for answering my questions on number theory and totally positive units. Special thanks go to Shingo Taki for pointing out two mistakes in an earlier version which led to two missing cases in the classification.
\bibliographystyle{JHEPsort}
\bibliography{literature}{}
\end{document}